\documentclass[a4paper, 12pt]{amsart}
\usepackage[ngerman, english]{babel}
\usepackage{amsfonts,amsmath}
\usepackage{amsthm}
\usepackage{txfonts}
\usepackage{faktor}
\usepackage[arrow, matrix, curve]{xy}

\newenvironment{remark}[1][Remark]{\begin{trivlist}
\item[\hskip \labelsep {\bfseries #1}]}{\end{trivlist}}

\newtheorem{theorem}{Theorem}[section]
\newtheorem{lemma}[theorem]{Lemma}
\newtheorem{proposition}[theorem]{Proposition}
\newtheorem{corollary}[theorem]{Corollary}
\newtheorem{definition}[theorem]{Definition}

\newcommand{\Rn}{\mathbb{R}^n}
\newcommand{\defi}{\coloneqq}
\newcommand{\C}{\mathbb C}
\newcommand{\R}{\mathbb{R}}
\newcommand{\Z}{\mathbb{Z}}
\newcommand{\N}{\mathbb{N}}
\newcommand{\Sph}{\mathbb{S}^1}
\newcommand{\W}{\mathcal{W}}
\newcommand{\Szw}{\mathbb{S}^2}
\newcommand{\Quat}{\mathbb H}
\newcommand{\ord}{\operatorname{ord}}
\newcommand{\conf}{\operatorname{conf}}
\newcommand{\diam}{\operatorname{diam}}
\newcommand{\abs}[1]{\lvert#1\rvert}
\def\a#1{\left\llbracket{#1}\right\rrbracket}

\newcommand{\QR}[2]{{\left.\raisebox{.2em}{$#1$}\middle/\raisebox{-.2em}{$#2$}\right.}}

\title[Klein bottles]{Existence of Minimizing Willmore Klein bottles in euclidean four-space}

\author[P. Breuning]{Patrick Breuning}
\address[Patrick Breuning]{Pastor-Felke-Str. 1, 76131 Karlsruhe}
\email{pbreuning@yahoo.de}

\author[J. Hirsch]{Jonas Hirsch}
\address[Jonas Hirsch]{Karlsruhe Institute of Technology, Institute for Analysis, Englerstr. 2, 76131 Kalrsruhe, Germany}
\email{jonas.hirsch@kit.edu}

\author[E. M\"ader-Baumdicker]{Elena M\"ader-Baumdicker}
\address[Elena M\"ader-Baumdicker]{Karlsruhe Institute of Technology, Institute for Analysis, Englerstr. 2, 76131 Kalrsruhe, Germany}
\email{elena.maeder-baumdicker@kit.edu}

\begin{document}


\begin{abstract}
Let $K = \R P^2 \sharp \R P^2$ be a Klein bottle. 
We show that the infimum of the Willmore energy among all immersed Klein bottles $f:K\to\R^n$, $n\geq 4$, is attained by a smooth embedded Klein bottle. We know from \cite{Massey, Hirsch} that there are three distinct regular homotopy classes of immersions $f:K\to\R^4$ each one containing an embedding. One is characterized by the property that it contains the minimizer just mentioned.
 For the other two regular homotopy classes we show $\W(f)\geq 8\pi$.
We give a classification of the minimizers of these two regular homotopy classes. In particular, we prove the existence of infinitely many distinct  embedded Klein bottles in $\R^4$ that have Euler normal number $-4$ or $+4$ and Willmore energy $8\pi$. The surfaces are distinct even when we allow conformal transformations of $\R^4$.
As they are all minimizers in their regular homotopy class they are Willmore surfaces.
\end{abstract}

\maketitle

 \section{Introduction}
 
For a two-dimensional manifold $\Sigma$ immersed into $\R^n$ via $f:\Sigma \to\R^n$, the Willmore energy is defined as
\begin{align*}
 \W(f)\defi\frac{1}{4} \int \abs{H}^2 d \mu_g,
\end{align*}
where $H$ is the mean curvature vector of the immersed surface, i.e.\ the trace of the second fundamental form. Integration is due to the area measure with respect to the induced metric $g= f^\ast \delta_{\text{eucl}}$. 

In this paper, we consider closed non-orientable manifolds $\Sigma$ of (non-orientable) genus $p=1,2$, i.e.\ our surfaces are of the type of $\R P^2$ or $K\defi\R P^2 \sharp \R P^2$ (a Klein bottle). We are interested in the existence and the properties of immersions 
 $f:\Sigma\to\R^n$, $n\geq 4$, that are regularly homotopic to an embedding and that have low Willmore energy.

Concerning a lower bound on the Willmore energy, a result of Li and Yau is very useful for closed surfaces immersed into $\R^n$ (see~\cite{LiYau}): Let $x\in\R^n$ be a point and $\theta(x)\defi |\{z\in \Sigma: f(z) = x\}|$ the (finite) number of distinct pre-images of $x$. Then $$\W(f)\geq 4\pi \theta(x).$$ 
As any immersed $\R P^2$  in $\R^3$ has at least one triple point \cite{Banchoff} it follows that $\W(f)\geq 12 \pi$ for any such immersion. Equality holds for example for Boy's surface, see~\cite{Kusner}. 
Similarly, as an immersed Klein bottle in $ \R^3$ must have double points we have that $\W(f)\geq 8\pi$ for such immersions. Kusner conjectured that Lawson's minimal Klein bottle in $\mathbb S^3$ is (after inverse stereographic projection) the minimizer of the Willmore energy for all Klein bottles immersed into $\R^3$, see \cite{Kusner, Lawson}. This immersion has energy about $9.7\pi$.

Since any $m$-dimensional manifold can be embedded into $\R^{2m}$ (\cite{Whitney})  it is natural to ask what is known about $\R P^2$'s and Klein bottles immersed into $\R^4$. Li and Yau showed that $\W(f)\geq 6\pi$ for any immersed $\R P^2$ in $\R^4$, and equality holds if and only if the immersion is the Veronese embedding \cite{LiYau}. It turns out that the Veronese embedding and the reflected Veronese embedding are representatives of the only two distinct regular homotopy classes of immersions containing an embedding. The number of regular homotopy classes is due to Whitney and Massey \cite{Massey} and Hirsch \cite{Hirsch}, see Section~\ref{Section2}.

 As in the case of $\R P^2$ we can count the number of distinct regular homotopy classes of immersions of a Klein bottle containing an embedding. There are three of them. 
 By a gluing construction of Bauer and Kuwert there is a Klein bottle embedded in $\R^4$ with Willmore energy strictly less than $8\pi$, see \cite[Theorem~1.3]{BauerKuwert}. We repeat parts of this gluing construction in Section~\ref{Section3} and conclude that this gives a Klein bottle in the regular homotopy class characterized by Euler normal number zero. As we can add arbitrary dimensions this construction yields an embedded Klein bottle $f:K\to\R^n$, $n\geq 4$, with $W(f)<8\pi$.
 It follows that the infimum of the Willmore energy among all immersed Klein bottles is less than $8\pi$. E.~Kuwert and Y.~Li  proved in \cite{KuwertLi} a compactness theorem for so called $W^{2,2}$-conformal immersions and a theorem about the removability of point singularities. With these methods we prove that the infimum among immersed Klein bottles is attained by an embedding. We know that the minimizer is smooth by the work of T.~Rivi{\`e}re \cite{Riviere2, Riviere}. Note that T.~Rivi{\`e}re proved independently a compactness result similar to the one of Kuwert and Li mentioned above, see \cite[Theorem~III.1]{Riviere}.\\
 The existence of the minimizer among immersed Klein bottles gives a partial answer to a question that was stated by F.\ Marques and A.\ Neves in \cite[Section~4]{MarquesNeves}: They asked about the infimum of the Willmore energy in $\R^3$ or $\R^4$ among all non-orientable surfaces of a given genus or among all surfaces in a given regular homotopy class and they asked whether it is attained.  
  Here is the first existence result: 
  
 \begin{theorem} \label{thm0}
  Let $S$ be the class of all immersions $f:\Sigma \to\R^n$ where $\Sigma$ is a Klein bottle. Consider
 \begin{align*}
  \beta_2^n\defi \inf \left\{\W(f): f\in S\right\}.
 \end{align*}
Then we have that $\beta_2^n< 8\pi $ for $n\geq 4$. Furthermore, $\beta_2^n$ is attained by a smooth embedded Klein bottle for $n\geq 4$.
\end{theorem}

We want to point out that the upper bound $\beta_2^n< 8 \pi$ can be improved. Let $\tilde{\tau}_{3,1}$ be the bipolar surface of Lawson's $\tau_{3,1}$-torus \cite{Lawson}. It is an embedded minimal Klein bottle in $\mathbb S^4$. After stereographic projection one obtains a Klein bottle $f: K \to \R^4$ with Willmore energy $\W(f)= 6 \pi \operatorname E \left(\frac{2\sqrt{2}}{3}\right)$ \cite{Lapointe}. Here, $\operatorname E(.)$ is the complete elliptic integral of second kind. We conclude that  $\beta_2^n \le  6 \pi \operatorname E \left(\frac{2\sqrt{2}}{3}\right) \approx 6.682 \pi< 8 \pi$. 
There is some indication that $\tilde{\tau}_{3,1}$ is the actual minimizer among immersed Klein bottles in $\R^4$, compare the forthcoming paper \cite{we}.\\

%
We will show in Section~\ref{Section2} that immersions in one of the other two regular homotopy classes of immersed Klein bottles in $\R^4$ satisfy $\W(f)\geq 8\pi$. There are minimizing representative embeddings $f_i: K\to\R^4$, $i=1,2$ with  Euler normal number $-4$ for $f_1$ and $+4$ for $f_2$ (for the definition of the Euler normal number, see Section~\ref{Section2}). 

We prove the following:
\begin{theorem} \label{thm1}
There is a one parameter family of smooth embedded Klein bottles $f^r_i:K \to\R^4$, $i=1,2$, $r\in\R^+$, with $W(f^r_i)=8\pi$ for $i=1,2$. The embeddings $f^r_1$ have Euler normal number $e(\nu)=-4$. The oriented double cover of the surfaces $\tilde f^r_1: M_r\to\R^4$ are conformal, where $M_r = \QR{\C}{\Gamma_r}$ is the torus generated by $(1,ir)$. Furthermore, $\tilde f^r_1$ are twistor holomorphic.
 The second embeddings $f^r_2$ are obtained by reflecting $f^r_1(K)$ in $\R^4$, and they have Euler normal number  $+4$. Every embedding $f^r_1,f^r_2$ is a minimizer of the Willmore energy in its regular homotopy class. Thus, all discovered surfaces are Willmore surfaces.

 For $r\neq r'$ the surfaces $f^r_1(K)$ and $f^{r'}_1(K)$ are different in the following sense: For all conformal transformations $\Phi$ of $\R^4$ we have  $f^r_1(K)\neq \Phi \circ f^{r'}_1(K)$ for $r\neq r'$.
 
 Furthermore, there is a classification (including a concrete formula) of immersed Klein bottles in $\R^4$ that satisfy $\W(f)= 8\pi$ and $|e(\nu)| = 4$.
\end{theorem}
Our techniques can also be used for $\R P^2$'s with $W(f)=6\pi$. As such surface must be a conformal transformation of the Veronese embedding (\cite{LiYau}) we get an explicit formula for this surface:

\begin{proposition}
 Define $f:  \Szw \to\C^2 = \R^4$ by $f(z)= \left(\bar z \frac{\abs{z}^4 - 1}{\abs{z}^6 + 1}, z^2 \frac{\abs{z}^2 +1}{\abs{z}^6 + 1}\right)$. Then $f(\Szw)$ is the Veronese surface (up to conformal transformation of $\R^4$).
\end{proposition}

 We give an overview of the structure of this paper. In Section~\ref{Section1} we prove that each torus carrying an antiholomorphic involution without fixpoints is biholomorphically equivalent to a torus $T$ with a rectangular lattice generated by $(1,\tau)$. On $T$, the involution has the form $I(z) = \bar z + \frac{1}{2}$ up to M\"obius transformations on $T$. Section~\ref{Section2} contains the proof in the non-orientable case of the so called ``Wintgen inequality'' which is $\W(f)\geq 2\pi (\chi + |e(\nu)|)$, see \cite{Wintgen}. We then give an introduction to the theory of twistor holomorphic immersions into $\R^4$ (see \cite{Friedrich}) and construct the surfaces of Theorem~\ref{thm1} with this theory. The same methods yield the formula for the Veronese embedding. 
 We explain in Section~\ref{Section3} that the gluing construction of Bauer and Kuwert \cite{BauerKuwert} gives an embedded Klein bottle $f:K\to\R^n$, $n\geq 4$, with Willmore energy strictly less that $8\pi$ (thus, with Euler normal number zero if $n=4$). This embedding is not in one of the regular homotopy classes of the embeddings of Theorem~\ref{thm1}. 
 After this, we show that a sequence of Klein bottles $f_k :K \to\R^n$ where the oriented double covers diverge in moduli space satisfies $$\liminf_{k\to\infty}W(f_k)\geq 8\pi.$$ We use this estimate together with techniques and results from \cite{KuwertLi, Riviere, Riviere2} to show Theorem~\ref{thm0}.
 
 \begin{remark}
  In $\R^3$, there is no immersed Klein bottle with Willmore energy $8\pi$. If it existed then we could invert at one of the double points in $\R^3$. We would get a complete minimal immersion in $\R^3$ with two ends.
  But due to \cite{Kusner} this surface must be embedded, a contradiction.  
 \end{remark}

\section*{Acknowledgment}
First of all we want to thank Ernst Kuwert for the initial idea. He posed the existence of Willmore minimizing Klein bottles as an open problem and proposed the gluing of two Veronese surfaces as a first step to obtain a competitor below $8\pi$. Furthermore, he gave us a sketch of the proof of Theorem~\ref{thm0}. Secondly we want to thank Tobias Lamm for drawing our attention to this problem and the helpful discussions.

%
 
 \section{Antiholomorphic involutions on the torus}\label{Section1}
 
Let $N$ be a non-orientable manifold of dimension two and $\tilde f: N\to\Rn$ ($n \geq 3$) an immersion. 
We equip $N$ with the induced Riemannian metric $\tilde f^\ast \delta_{\text{eucl}}$.
Consider $q:M\to N$, the conformal oriented 
two-sheeted cover of $N$, and define $f\defi \tilde f\circ q$. 
As every $2-$dimensional 
oriented manifold can be locally conformally reparametrized $M$ is a Riemann surface that is conformal to $(M, f^\ast \delta_{\text{eucl}})$.
Let $I:M\to M$ 
be the antiholomorphic order two deck transformation for $q$. The map $I$ is an 
antiholomorphic involution without fixpoints such that $ f\circ I = f$. 

Now consider the situation where $N$ is the Klein bottle, i.e.\ $N$ is compact, without boundary and has non-orientable genus two. In this case, the oriented two-sheeted cover $q: T^2 \to N$ lives on the two-dimensional torus $T^2$. 
It is the aim of this section to classify all antiholomorphic involutions without fixpoints on a torus $T^2$ up to M\"obius transformation. A M\"obius transformation is a biholomorphic map $\varphi: T^2\to T^2$. 
We use the fact that every torus is a quotient space $\QR{\C}{\Gamma}$ where $\Gamma$ is a lattice in $\C$, i.e.\
\begin{align*}
 \Gamma =\{m\omega + n \omega': m,n\in\mathbb Z\}
\end{align*}
where $\omega,\omega'\in \C=\mathbb R^2$ are vectors that are linearly independent over $\R$.  We call $(\omega,\omega')$ a \emph{generating pair} of $\Gamma$.

\begin{theorem} \label{thminvolution}
Consider a lattice $\Gamma$ in $\C$ generated by a pair $(1,\tau)$ where $\Im(\tau)>0$,  $-\frac{1}{2}<\Re(\tau)\leq \frac{1}{2}$ and $|\tau|\geq 1$.
Let $I: \QR{\C}{\Gamma} \to  \QR{\C}{\Gamma}$ be an antiholomorphic involution without fixpoints. Then $\Gamma$ must be a rectangular lattice, i.e.\ $\tau\in i\R^+$ and, up to M\"obius transformation, the induced doubly periodic map $\hat I: \C\to\C$ is of following form:
\begin{align*}
\text{Either } \ \hat I(z)=\bar z + \tfrac{1}{2} \ \text{ or  } \ \hat I(z)= -\bar z + \tfrac{\tau}{2}.
\end{align*}

%

\end{theorem}

\begin{remark}
 \begin{enumerate}
 \item A similar result can be found in \cite[Appendix~F]{KusnerSchmitt}. For the sake of completeness we give a full proof of Theorem~\ref{thminvolution} in the following. The case that $\Gamma$ is a hexagonal lattice, i.e.\ generated by $(1,e^{i\frac{\pi}{3}})$, and $\alpha\Gamma = \bar \Gamma$ with $\alpha =e^{li\frac{\pi}{3}}$, $l=1,2,4,5$  is not considered in the proof of \cite{KusnerSchmitt}.
  \item The expression ``up to M\"obius transformation'' means that there is a M\"obius transformation $\varphi:  \QR{\C}{\Gamma} \to  \QR{\C}{\Gamma}$ such that $\varphi^{-1}\circ I \circ \varphi$ is of the claimed form. 
  When we have an antiholomorphic involution $I$ without fixpoints on a torus $\QR{\C}{\Gamma}$ then $\varphi^{-1}\circ I\circ\varphi$ is also an antiholomorphic involution without fixpoints on that torus. 
  Therefore, it only makes sense to classify such involutions up to M\"obius transformation.
  \item  Every map $I:\QR{\C}{\Gamma}\to \QR{\C}{\Gamma}$ induces a map $\hat I: \C\to\C$ that is doubly periodic with respect to $\Gamma$. From now on we denote $\hat I$ simply by $I$.
 \end{enumerate}

\end{remark}

We prove this theorem in several steps. But at first we explain how we come to the case of a general lattice.

\begin{proposition} \label{biholom}
Let $\Gamma$ be a lattice in $\C$. Then there exists a generating pair $(\omega,\omega')$ such that $\tau\defi\frac{\omega'}{\omega}$ satisfies $\Im(\tau)>0$, $-\frac{1}{2}<\Re(\tau)\leq \frac{1}{2}$, $|\tau|\geq 1$ and if $|\tau|=1$ then $\Re(\tau)\geq 0$. Let $\tilde\Gamma$ be the lattice generated by $(1,\tau)$. Then there exists a biholomorphic map $\varphi:\QR{\C}{\Gamma}\to \QR{\C}{\tilde\Gamma}$.
\end{proposition}

\begin{proof}
The pair $(\omega,\omega')$ is sometimes called ``canonical basis''. The proof of the existence of this basis can be found in \cite[Chapt.~7, Theorem 2]{Ahlfors}. For the biholomorphic map we define $\tilde\varphi(z)\defi \frac{z}{\omega}$ for $z \in\C$. Then $\varphi\left([z]\right)\defi \tilde\varphi(z)$, $[z]\in\QR{\C}{\Gamma}$ defines a biholomorphic map $\varphi:\QR{\C}{\Gamma}\to\QR{\C}{\tilde \Gamma}$

\end{proof}

\begin{lemma} \label{lemma1}
 Let $\Gamma$ be a lattice in $\C$ and $I: \QR{\C}{\Gamma} \to  \QR{\C}{\Gamma}$ an antiholomorphic involution. 
 Then $I$ is of the form $I(z)= a \bar z + b$ where $a,b \in\C$ with $a \bar\Gamma =\Gamma$, $|a|=1$ and $a \bar b + b \in\Gamma$. Here, $\bar\Gamma$ is the complex conjugation of $\Gamma$.
\end{lemma}

\begin{proof} 

 Define $\psi( z)\defi I(\bar z)$. Notice  $\psi: \QR{\C}{\bar \Gamma} \to \QR{\C}{\Gamma}$ is holomorphic. Let $\Gamma$ be generated by $(\tau_1,\tau_2)$. The derivative $\psi': \C \to\C$ is holomorphic and bounded on the compact fundamental domain $F\defi \{t_1\tau_1 + t_2 \tau_2: 0\leq t_1,t_2 \leq 1\}$. The periodicity of $\psi'$ implies that it is bounded in all of $\C$. By Liouville's theorem we get that $\psi' = a$ for an $a\in\C$. Therefore, we have that $\psi(z) = a z + b$ for a vector $b\in\C$. By $I:\QR{\C}{\Gamma} \to  \QR{\C}{\Gamma}$ we have that
 \begin{align*} 
 \Gamma \ni I(z + \omega) - I(z)= \psi(\bar z + \bar\omega) - \psi(\bar z)= a\bar\omega  \ \ \ \ \forall \omega\in\Gamma, 
 \end{align*}
which implies $a\bar\Gamma \subset \Gamma$. For the other implication we use that $\psi$ is one-to-one (if restricted to the fundamental domain $F$). The map $\Phi\defi\bar I$ is an inverse of $\psi$ because $\bar I\circ \psi(z)=\bar I\circ I(\bar z)= z \mod \bar\Gamma$ and $\psi\circ \bar I(z) = I(I(z))=z \mod \Gamma$.
%
The same argument as above implies that there are complex numbers $c,d\in\C$ such that $\Phi(z)=c z + d$. 
So we have that 
 \begin{align*} 
 \bar \Gamma \ni \bar I(z + \omega) - \bar I(z)= \Phi( z + \omega) - \Phi( z)= c\omega  \ \ \ \ \forall \omega\in\Gamma, 
 \end{align*}
which implies $c\Gamma\subset \bar\Gamma$. We get that
\begin{align*}
 \operatorname{id}\big|_{\QR{\C}{\Gamma}}(z) = \psi\circ\Phi(z) = acz + ad + b
\end{align*}
which implies $ac=1$ and $\frac{1}{a}\Gamma\subset \bar\Gamma$.
\end{proof}

\begin{lemma}\label{lemma2}
Let $\Gamma$ be a lattice in $\C$ generated by $(1,\tau)$ with $\Im(\tau)>0$. Then all M\"obius transformations $\varphi:\QR{\C}{\Gamma}\to \QR{\C}{\Gamma}$ are of the form $\varphi(z)= \alpha z + \delta$ with $\delta \in \C$ and
\begin{enumerate}
 \item if $\tau=i$ (quadratic lattice) then $\alpha\in\{1,-1,i,-i\}$,
 \item if $\tau= e^{i\frac \pi 3}$ or $\tau=e^{2i\frac{\pi}{3}}$ (hexagonal lattice) then $\alpha\in\{e^{li\frac{\pi}{3}}: l=1,...,6\}$,
 \item if $\Gamma$ is neither the quadratic lattice nor the hexagonal lattice then $\alpha \in \{1,-1\}.$
\end{enumerate}
\end{lemma}

\begin{proof}
 At first we note that a translation $\varphi(z)=z + \delta$ for a $\delta\in\C$ is always a M\"obius transformation. 
 Therefore, we assume that $\varphi(0)=0$ (by composing with a translation). The rest of the proof can be found in \cite[Chapt.~III, Proposition~1.12.]{Miranda}.
\end{proof}

\begin{lemma} \label{lemma3}
Let $\Gamma$ be a lattice in $\C$ generated by $(1,\tau)$ with $\Im(\tau)>0$ and $|\tau|=1$. Let $I$ be an antiholomorphic involution on $\QR{\C}{\Gamma}$ of the form $I(z)= a \bar z + b$ with $a\in\{+\tau, - \tau\}$. Then $I$ has a fixpoint.
\end{lemma}

\begin{proof}
Let $\varphi(z)= z + \delta$ be a translation on $\QR{\C}{\Gamma}$. We have that
\begin{align*}
 \varphi^{-1}\circ I\circ \varphi (z)= a (\bar z + \bar\delta) + b - \delta= a\bar z + b + a \bar \delta - \delta.
\end{align*}
Consider now a translation with $\delta\in\R$. Then we have that
\begin{align*}
 \tilde I(z)\defi  \varphi^{-1}\circ I\circ \varphi (z) = \pm\tau\bar z + (\pm\tau -1)\delta + b.
\end{align*}
By $\Im (\pm\tau - 1)\neq 0$ we can choose $\delta\in\R$ such that $(\pm\tau - 1)\delta + b \in\R.$
Hence by passing from $I$ to $\tilde I$ we can assume that the involution is of the form $I(z)= \pm\tau \bar z + b$ with $b\in\R$. 
Lemma~\ref{lemma1} implies that $\pm\tau b + b \in\Gamma$. Since $(1,\tau)$ is the generating pair of $\Gamma$ we get that $b\in\Z$. But $I(z)=\pm\tau\bar z + n$ with $n\in\Z$ has the fixpoint $0$. Then the original involution also had a fixpoint. 
\end{proof}

\begin{lemma}\label{lemma4}
Let $\Gamma$ be a lattice in $\C$ generated by $(1,\tau)$ with 
$\Im(\tau)>0$ and $|\tau|=1$. Let $I$ be an antiholomorphic involution on $\QR{\C}{\Gamma}$ of the form $I(z)= a \bar z + b$ with $a \not \in\{+1, - 1\}$. Then $I$ has a fixpoint.
 
\end{lemma}

\begin{proof}
 Since $a$ satisfies $a\bar\Gamma =\Gamma$ and $|a|=1$ (cf. Lemma~\ref{lemma1}) we want to know how many lattice points lie on the unit circle $\mathbb{S}^1$. There are two cases.\\
 Case 1: $\tau-1\not\in\Sph$. But $|\tau-1|^2\neq 1$ is here equivalent to $\Re(\tau)\neq \frac{1}{2}$ since $|\tau-1|^2=2-2 \Re (\tau)$. 
 Therefore we know that $\Gamma$ cannot be the hexagonal lattice and there are exactly four lattice points on $\Sph$, namely $1,-1,\tau$ and $-\tau$. Since $a\bar\Gamma=\Gamma$ and $1\in\bar\Gamma$ we have that $a\in\Gamma\cap\Sph$, which implies $a\in\{1,-1,\tau,-\tau\}$. But $a \not \in\{+1, - 1\}$ is an assumption and $a\in\{\tau,-\tau\}$ implies that $I$ has a fixpoint by the previous lemma.\\
 Case 2: $\tau-1\in\Sph$. This corresponds to the hexagonal lattice, $\tau= e^{i\frac{\pi}{3}}$. There are six lattice points lying on $\Sph$, namely $e^{li\frac{\pi}{3}}, l=1,...,6$. 
 Again as in the first case we have that $a\in\Gamma\cap\Sph$. The cases $l=3$ and $l=6$ are not possible by assumption, therefore we get that $a\in \{\tau^l:l=1,2,4,5\}$. Now consider a M\"obius transformation of the hexagonal lattice $\varphi(z)=\alpha z$ with $\alpha\neq0$. Lemma~\ref{lemma2} yields $\bar\alpha\in\{\tau^k:k=1,...,6\}$. We compose
 \begin{align*}
  \tilde I(z)\defi\varphi^{-1}\circ I\circ \varphi (z) = \frac{\bar\alpha}{\alpha} a  \bar z + \frac{b}{\alpha}=\tau^{2k + l}\bar z + \bar\alpha b.
 \end{align*}
If $l$ is even, then we choose $k$ such that $2k + l =6$. Thus, we are in the case $a=1$. If $l=5$ then we compose with the M\"obius transformation $\varphi(z)=\alpha z$ where $\alpha = \tau^{4}$ (which is equivalent to $k=-2$). We have then reduced it to the case $a=\tau$, which is Lemma~\ref{lemma3}.

\end{proof}

\begin{lemma} \label{lemma5}
 Consider a lattice $\Gamma$ in $\C$ generated by a pair $(1,\tau)$ with $\Im(\tau)>0$, $-\frac{1}{2}<\Re(\tau)\leq \frac{1}{2}$ and $|\tau|> 1$.
 Let $I:\QR{\C}{\Gamma}\to\QR{\C}{\Gamma}$ be an antiholomorphic involution. Then we have that $I(z)=a\bar z + b$ with $a\in\{-1,1\}$.
\end{lemma}

\begin{proof}
 By Lemma~\ref{lemma1} we know that $a\bar\Gamma=\Gamma$ and $|a|=1$. Hence $a\in\Sph\cap\Gamma$. We claim that $\Sph\cap\Gamma=\{-1,1\}$. Since $|\tau|>1$, we know that $\pm\tau\not\in\Sph\cap\Gamma$. But then we only have to consider the case that $z\in \Sph\cap\Gamma$ is of the form $z=-1 + l\tau$ for an $l\in\Z\setminus\{0\}$. We use the assumptions on $\tau$ and get 
 \begin{align*}
  |-1 + l\tau|^2= 1 + l^2|\tau|^2 - 2l\Re(\tau)> 1 + l^2 - l \geq 1.
 \end{align*}
This strict inequality shows the lemma.
\end{proof}

\begin{definition}
 A lattice $\Gamma$ in $\C$ is called a \emph{real lattice} if it is stable under complex conjugation, i.e.\ $\bar\Gamma=\Gamma$.
\end{definition}

\begin{lemma}\label{lemma6a}
 Let $\Gamma$ be a real lattice generated by $(1,\tau)$ with $-\frac{1}{2}<\Re(\tau)\leq \frac{1}{2}$. Then we have that $\Re(\tau)\in\{0,\frac{1}{2}\}$.
\end{lemma}
\begin{proof}
 Let $\tau=x+iy$. Then there are $m,n\in\Z$ such that
 \begin{align*}
  & \bar \tau =x-iy =n + m(x+iy)\\
  \Leftrightarrow \ \ & x-n-mx - i(my+y)=0\\
  \Leftrightarrow \ \ & m=-1\ \ \text{ and } \ \ x(1-m)=n.
 \end{align*}
 This implies that $\Re(\tau)=x\in (-\frac{1}{2},\frac{1}{2}]\cap \{\frac{n}{2}:n\in\Z\}=\{0,\frac{1}{2}\}$.
\end{proof}

\begin{lemma} \label{lemma6}
 Let $\Gamma$ be a lattice generated by $(1,\tau)$ with $-\frac{1}{2}<\Re(\tau)\leq \frac{1}{2}$ and let $I(z)=a\bar z +b$ be an antiholomorphic involution with $a=-1$. Then the lattice is real and $\Re(\tau)\in\{0,\frac{1}{2}\}$. If $\Re(\tau)=\frac{1}{2}$ then $I$ has a fixpoint. If $\Re(\tau)=0$ then $I(z)=-\bar z + \frac{\tau}{2}$ (up to M\"obius transformation) and $I$ has no fixpoints.
\end{lemma}

\begin{proof}
 As every lattice satisfies $-\Gamma=\Gamma$ we have by Lemma~\ref{lemma1} that $\bar\Gamma=-\bar\Gamma=\Gamma$, i.e.\ the lattice is real. The previous lemma yields $\Re(\tau)\in\{0,\frac{1}{2}\}$.\\

 Case $\Re(\tau)=\frac{1}{2}$: We note that 
 \begin{align}\label{Im}
  i\,\R\cap\Gamma=\{2m i\Im(\tau): m\in\Z\}.
 \end{align}
By composing $I$ with a translation we can assume that $b\in i\R$: 
Consider the translation $\varphi(z)= z + \delta$, then 
\begin{align*}
 \tilde I(z)\defi\varphi^{-1}\circ I\circ \varphi (z)= -\bar z + b - \bar \delta - \delta = -\bar z + b -2\Re(\delta).
\end{align*}
 Thus, we can subtract the real part of $b$ and consider $\tilde I$ instead of $I$. \\
 But by $a\bar b+b\in\Gamma$ (cf. Lemma~\ref{lemma1}) and (\ref{Im}) we have that $2 b=-\bar b + b\in\Gamma$ and $2b= 2m i\Im(\tau)$ for an $m\in\Z$. Composing the involution with another translation yields that $b=m\tau$ for an $m\in\Z$. Hence $I(z)=-\bar z + m\tau$ which has the fixpoint $0$. \\
 
 Case $\Re(\tau)=0$: Here, $\Gamma$ is a rectangular lattice. By translation as in the first case  we assume $b\in i\,\R$. 
 Therefore, we get that $-\bar b + b = 2b \in \Gamma\cap i\,\R=\{m\tau:m\in\Z\}$, hence $b=\frac{m}{2}\tau$ for an $m\in\Z$. Observe that $m$ cannot be even because otherwise I would have a fixpoint. As the formula for $I$ is only defined modulo $\Gamma$, we have that $I(z)=-\bar z + \frac{\tau}{2}$. We only have to show that this $I$ has no fixpoint: An equality like
 \begin{align*}
  I(z)-z= -\bar z - z + \frac{\tau}{2}= -2\Re(z) + \frac{\tau}{2} =n + m\tau
 \end{align*}
 cannot hold for numbers $m,n\in\Z$ because $\tau$ is purely imaginary.
\end{proof}

\begin{lemma}\label{lemma7}
Let $\Gamma$ be a lattice generated by $(1,\tau)$ with $-\frac{1}{2}<\Re(\tau)\leq \frac{1}{2}$ and let $I(z)=\bar z +b$ be an antiholomorphic involution. Then, up to M\"obius transformation, $I$ is of the form $I(z)=\bar z + \frac{1}{2}$ and the lattice satisfies $\Re(\tau)\in\{0,\frac{1}{2}\}$. If $\Re(\tau)=\frac{1}{2}$ then $I$ has fixpoints
\end{lemma}

\begin{proof}
 By composing with a translation $\varphi(z)=z + \delta$ we get
 \begin{align*}
 \tilde I(z)\defi\varphi^{-1}\circ I\circ \varphi (z)= \bar z + b + \bar \delta - \delta = -\bar z + b -2i\Im(\delta).
\end{align*}
Thus, we can assume that $b\in\R$. Now we have that $2 b= \bar b + b\in\Gamma\cap \R=\Z$ and therefore $b=\frac{m}{2}$ for an $m\in\Z$. If $m$ was even, then $I$ would have the fixpoint $0$, and since the formula is only defined modulo $\Gamma$ we have that $b=\frac{1}{2}$. As $a\bar\Gamma =\Gamma$ with $a=1$ we know that the lattice is real and hence satisfies $\Re(\tau)\in\{0,\frac{1}{2}\}$ (Lemma~\ref{lemma6a}).
It remains to check in which cases $I$ has fixpoints: Let $m,n\in \Z$. If 
\begin{align*}
 I(z)-z = \bar z - z + \frac{1}{2} = -2i\Im(z) + \frac{1}{2} = n + m\Re(\tau) + m i\Im(\tau)
\end{align*}
then $\Im(z)=-\frac{m}{2}\Im(\tau)$ and $\Re(\tau)=\frac{1-2n}{2m}$. Hence if the real part of $\tau$ is an odd number divided by an even number, then $I$ has a whole line of fixpoints, otherwise it has no fixpoints. 
\end{proof}

We are now able to prove Theorem~\ref{thminvolution}:
\begin{proof}
 Any involution is of the form $I(z)=a\bar z + b$ by Lemma~\ref{lemma1}. If $|\tau|> 1$ then Lemma~\ref{lemma5} implies that $a\in\{1,-1\}$. 
 The case $a=-1$ is Lemma~\ref{lemma6} and the case $a=1$ is Lemma~\ref{lemma7}. If $|\tau|=1$ then we have that $a\in\{-1,1\}$ by Lemma~\ref{lemma4}. Lemma~\ref{lemma6} and Lemma~\ref{lemma7} apply also for this case.
\end{proof}


\section{Willmore surfaces of Klein bottle type in $\R^4$ with energy $8\pi$} \label{Section2}

Let $M$ be a closed manifold of dimension two (orientable or non-orientable) immersed into an oriented 4-dimensional Riemannian manifold $(X^4,h)$. 
The immersion induces a metric $g$ on $M$, a connection $\nabla$ on tangential bundle $TM$ and a connection $\nabla^\perp$ on the normal bundle  $NM$.
Since $TM$ and $NM$ are both two-dimensional their curvature operator is determined by scalars.
Let $\{E_1, E_2, N_1, N_2\}$ be an orthonormal oriented frame of $X^4$ in a neighborhood $U$ of $ x_0 \in M$ such that $E_1, E_2$ is a basis for $T_xM$ and $N_1, N_2$ a basis for $N_xM$ for all $x \in U$. 
The scalars of interest are the Gauss curvature given on $U$ by
\[ K(x)= R(E_1, E_2, E_2, E_1) = \langle \nabla_{E_1} \nabla_{E_2} E_2 - \nabla_{E_2}\nabla_{E_1} E_2 - \nabla_{[E_1, E_2]} E_2, E_1 \rangle, \] 
and the trace of the curvature tensor of the normal connection given on $U$ by
\[ K^\perp(x) := \langle R^\perp(E_1,E_2)N_2, N_1 \rangle = \langle \nabla^\perp_{E_1} \nabla^\perp_{E_2} N_2 - \nabla^\perp_{E_2}\nabla^\perp_{E_1} N_2 - \nabla^\perp_{[E_1, E_2]} N_2, N_1 \rangle. \]
Introducing the connection $1$-forms $\{ w_i^j\}_{i,j = 1,2,3,4} $ by
\begin{equation}\label{eq:1-forms} D_{v}e_i := w_i^j(v) e_j \text{ for } v \in T_xM\end{equation}
where $\{E_1, E_2, N_1, N_2\}=\{e_1,e_2,e_3,e_4\}$ and $D$ is the Levi-Civita connection of $X$.
Classical calculations show that 
\[ R(X,Y)E_2 = dw_2^1(X,Y)E_1  \text{ and } R^\perp(X,Y)N_2 = dw_4^3(X,Y) N_1, \]
hence the definition of  $K$ and $K^\perp$ is  independent of the orientation of $E_1, E_2$.\\
The Weingarten equation relates $D$ to the connection $\nabla$ and the second fundamental form $A(v,w) = (D_vw)^\perp$ for two vector fields $v,w$ on $M$ by $D_vw= \nabla_vw + A(v,w)$.
We can express $R$ and $R^\perp$ in terms of the second fundamental form and the curvature operator $R^X$ of the ambient manifold $X^4$ using $A_{ij}$ for $A(E_i,E_j)$
\begin{align}\label{eq:K}
K(x)&= R^X(E_1,E_2,E_2,E_1) + \langle (D_{E_1}E_1)^\perp, (D_{E_2}E_2)^\perp \rangle -\langle (D_{E_1}E_2)^\perp, (D_{E_1}E_2)^\perp \rangle \nonumber\\
&=    R^X(E_1,E_2,E_2,E_1)  + \langle A_{11}, A_{22} \rangle - \langle A_{12}, A_{12} \rangle.
\end{align}
Similarly, one gets for the normal curvature 
\begin{align*}
K^\perp(x)&= R^X(E_1,E_2,N_2,N_1) + \langle (D_{E_1}N_1)^\top, (D_{E_2}N_2)^\top \rangle -\langle (D_{E_1}N_2)^\top, (D_{E_2}N_1)^\top \rangle \\
&=    R^X(E_1,E_2,N_2,N_1)  +\left(  \sum_{j=1,2}\langle A_{1j},N_1\rangle \langle A_{2j}, N_2 \rangle - \langle A_{1j},N_2\rangle \langle A_{2j}, N_1 \rangle \right).
\end{align*}
Observe that the second part can be expressed as $\langle \sum_{j=1,2} A_{1j}\wedge A_{2j} , N_1 \wedge N_2 \rangle$. Introducing the trace free part $A_{ij}^\circ = A_{ij} - \frac{1}{2} H g_{ij}$ using $A^\circ_{11} + A^\circ_{22} =0, A_{12}=A^\circ_{12}$ the equation for $K^\perp$ simplifies to
\begin{equation}\label{eq:K^perp} K^\perp(x) =  R^X(E_1,E_2,N_2,N_1)  + 2 \langle A^\circ_{11}\wedge A^\circ_{12} ,N_1\wedge N_2 \rangle . \end{equation}
Recall that the Euler number of the normal bundle can be expressed similar to the Gauss-Bonnet formula \cite{Milnor} as
\begin{align}\label{hf}
 e(\nu) =\frac{1}{2\pi} \int_{M} K^\perp.
\end{align}
As a corollary of these calculations we obtain a classical inequality by Wintgen. This inequality was known to be true for oriented surfaces. We extend the result to non-orientable surfaces.

\begin{theorem}[Wintgen \cite{Wintgen}] \label{Wintgen}
Let $M$ be a closed manifold of dimension two (orientable or non-orientable) and Euler characteristic $\chi$. Consider an immersion $f:M\to \R^4$ and denote by $e(\nu)$ the Euler normal number of $f$. Then we have that
\begin{align}\label{Wintgenineq}
\W(f)\geq 2\pi (\chi+ |e(\nu)|)
\end{align} 
and equality holds if and only if 
\begin{equation}\label{eq:Wintgenineq}  \abs{A^\circ_{11}}^2 = \abs{A^\circ_{12}}^2, \langle A^\circ_{11}, A^\circ_{12} \rangle = 0 \text{ and $K^\perp$ does not change sign. }\end{equation}
\end{theorem}
\begin{proof}
 The proof for the orientable case can be found in \cite{Wintgen}. 
 Note that in this case $e(\nu) =2 I$ (see \cite{Lashof}), where $I$ is the self-intersection number due to Whitney, see \cite{Whitney}. And we have the equality $\chi = 2- 2p$, where $p$ is the genus of $M$. \\
 The general case follows from \eqref{eq:K}  and \eqref{eq:K^perp} and the flatness of $\R^4$. Equality \eqref{eq:K} becomes $K=\langle A_{11}, A_{22}\rangle - \langle A_{12}, A_{12}\rangle$ and so $\abs{ H}^2 = \abs{A}^2 + 2K$. Together with $\abs{A^\circ}^2 = \abs{A}^2 - \frac{1}{2}\abs{H}^2$ we have 
\[ \frac{1}{2} \abs{ H}^2 = 2K + \abs{A^\circ}^2. \]
Equation \eqref{eq:K^perp} becomes $K^\perp = 2 \langle A^\circ_{11}\wedge A^\circ_{12} ,N_1\wedge N_2 \rangle$ and we can estimate 
\begin{align*}
\abs{K^\perp} &=  2 \abs{ A^\circ_{11} \wedge A^\circ_{12}} \abs{N_1 \wedge N_2} = 2 \left( \abs{A^\circ_{11}}^2\abs{A^\circ_{12}}^2 - \langle A^\circ_{11}, A^\circ_{12} \rangle^2 \right)^{\frac12} \\
&\le 2 \abs{A^\circ_{11}}\abs{A^\circ_{12}} \le \abs{A^\circ_{11}}^2 + \abs{A^\circ_{12}}^2 = \frac12 \abs{A^\circ}^2
\end{align*}
with equality if and only if the first part of \eqref{eq:Wintgenineq} holds.
Combining both gives
\[\frac{1}{2} \abs{H}^2 = 2K + \abs{A^\circ}^2 \ge 2 K + 2 \abs{K^\perp} .\]
Multiplying by $\frac12$  and integrating over $M$ gives 
\[ \W(f) \ge \int_{M} K + \int_{M} \abs{K^\perp} \ge  \int_{M} K + \left| \int_{M} K^\perp \right|\]
with equality if and only if $K^\perp$ does not change sign.
\end{proof}


\begin{remark}
 As we are interested in the case $p=2$, i.e.\ $N= \R P^2\sharp \R P^2$ is a Klein bottle, the inequality above does not give us any information about the Willmore energy in the case $e(\nu)=0$. But we get information about the energy if the immersion is an embedding due to the following theorem.
\end{remark}

\begin{theorem}[Whitney, Massey \cite{Massey}]  \label{Whitney}
 Let $N$ be a closed, connected, non-orientable manifold of dimension two with Euler characteristic $\chi$. Consider an embedding $f:N\to \R^4$ with Euler normal number $e(\nu)$. Then $e(\nu)$ can take the following values:
 \begin{align*}
  -4 + 2\chi, 2\chi, 2\chi + 4, 2 \chi + 8,...,4-2\chi.
 \end{align*}
Furthermore, any of these possible values is attained by an embedding of $N$ into $\R^4$.
 
\end{theorem}

\begin{corollary}\label{cor:Wgeq8pi}
 Let $N= \R P^2\sharp \R P^2$ be a Klein bottle. Consider an immersion $f:N\to \R^4$ that is regularly homotopic to an embedding, and denote by $e(\nu)$ the Euler normal number of $f$. If $e(\nu)\neq 0$ then $\W(f)\geq 8\pi$.
\end{corollary}
\begin{proof}
Due to \cite[Theorem~8.2]{Hirsch} two immersions $f,g:N\to\R^4$ are regularly homotopic if and only if they have the same normal class. 
By assumption, the given immersion $f$ is regularly homotopic to an embedding $g:N\to\R^4$. Theorem~\ref{Whitney} and $\chi(N)=2-p=0$ implies that $e(\nu_f)=e(\nu_g)\in\{-4,0,4\}$. As $e(\nu)\neq 0$ we use Theorem~\ref{Wintgen} to see that $\W(f)\geq 8\pi$.
\end{proof}

\begin{remark}
In the case of genus one, we get by Theorem~\ref{Whitney} that the Euler normal number of the the Veronese embedding $f:\R P^2\to\R^4$  must be $e(\nu)\in\{-2,+2\}$. By the work of Hirsch \cite{Hirsch} we get that there are exactly two regular homotopy classes of surfaces of $\R P^2$-type containing an embedding.  Each regular homotopy class is represented by a Veronese embedding, one is the reflected surface of the other.
\end{remark}


For the construction of immersed Klein bottles with $\W(f)=8\pi$ and $e(\nu)\in\{-4,+4\}$ we need the theory of twistor holomorphic immersions. They were studied in \cite{Friedrich}, and we follow that paper.

\begin{definition} \label{twistorspace}
 Let $(X^4,h)$ be an oriented, $4-$dimensional Riemannian manifold. Consider a point $x\in X^4$ and let $P_x$ be the set of all linear maps $J:T_xX^4\to T_xX^4$ satisfying the following conditions:
 \begin{enumerate}
  \item $J^2=- Id$,
  \item $J$ is compatible with the metric $h$, i.e.\ $J$ is an isometry,
  \item $J$ preserves the orientation,
  \item defining the $2$-form $\Omega(t_1,t_2)\defi h( Jt_1,t_2)$, $t_1,t_2 \in T_xX^4$, then $ -\Omega \wedge \Omega$ equals the given orientation of $X^4$.
 \end{enumerate}
The set $P\defi \bigcup_{x\in X^4}P_x$ is a $\mathbb C P^1-$fiber bundle over $X^4$ (note $\QR{SO(4)}{U(2)}\cong\mathbb C P^1$). We call $P$ the \emph{twistor space of $X^4$} and denote by $\pi:P \to X^4$ the projection of the bundle.
\end{definition}

\begin{definition}[The lift of an immersion into the twistor space] \label{lift}
 Let $M$ be an oriented manifold of dimension two and $f:M\to X^4$ an immersion. We decompose the tangent space $T_{f(x)} X^4$ of the ambient manifold into the sum of the tangent space $T_x M$ and the normal space $N_x M$. 
 Let $E_1,E_2$ be a positively oriented orthonormal basis of $T_x M$ and $N_1,N_2$ an orthonormal basis of $N_x M$ such that $\{E_1,E_2,N_1,N_2\}$ is a positively oriented basis of $T_{f(x)} X^4$. We define the \emph{lift of the immersion $f$} by
 \begin{align}\label{eq:twisterconditions}
  F(x)&: T_{f(x)} X^4 \to T_{f(x)} X^4,\nonumber\\
  F(x)E_1&= E_2,\ \ 
  F(x)E_2 = - E_1, \\
  F(x) N_1 &= -N_2,\ \
  F(x) N_2 = N_1,\nonumber
 \end{align}
i.e.\ $F(x)$ is the rotation around the angle $\frac{\pi}{2}$ in the positive (negative) direction on $T_x M$ (on $N_x M$). In this way\footnote{The frame $\{E_1, E_2, N_1, N_2\}$ gives a local bundle chart of the pullback bundle $f^*P$ around $x$. The defined linear map $F(x): T_{f(x)}X^4 \to T_{f(x)}X^4$ is an element of the fiber $P_{f(x)}$. Hence we can either consider $F$ to be a map into the pullback bundle $f^*(P)$ being the identity on $M$ or as a map into $P$ by $\pi\circ F(x):=f(x)$. We follow the classical line and think of $F$ as a map into $P$.}, $F: M \to P$ is a lift of $f$.
\end{definition}

\begin{definition}[Twistor holomorphic] 
There exists an almost complex structure $Y$ on $P$ coinciding with the canonical complex structure on the fibers $\QR{SO(4)}{U(2)}\cong\C P^1$. For a point $J \in P$ the horizontal part $T^H_JP$ is determined by the Levi-Civita connection on $X^4$ and the complex structure on it is $d\pi^{-1}Jd\pi$, \cite[Section~1]{Friedrich}.
The pair $(P,Y)$ is a complex manifold if and only if the manifold $X^4$ is self-dual, see~\cite{Atiyah}. Let $M$ be an oriented two-dimensional manifold and $f:M\to X^4$ an immersion. 
 Denote by $I: T_x M \to T_x M$ the complex structure of $M$ with respect to the induced metric $f^{\ast} h$. The immersion $f$ is called \emph{twistor holomorphic} if the lift $F: (M,I)\to (P,Y)$ is holomorphic, i.e.\ $dF(I(t))= Y(dF(t))$.
\end{definition}

\begin{remark}
The couple $(M,I)$ from definition above is a Riemann surface and $I$ only depends on the conformal class of $f^*h$. The map $F$ has the property that for any conformal coordinates $\varphi: U \subset \C \to M$ the map $F\circ \varphi$ is holomorphic. Furthermore, the metric $( f\circ \varphi)^*h$ is conformal to the standard metric on $\C$. \\
On the other hand, if a the lift $F: M \to P$ of a map $f:M \to X^4$ from a Riemann surface $M$ has the property that $F\circ \varphi: U \to P$ is holomorphic for any conformal coordinates $\varphi: U \to M$ it is not hard to check that $F$ is twistor holomorphic as defined above. In fact, this is the picture we will use in the following. 
\end{remark}

\begin{remark}
 As we only want to use the construction of twistor spaces for $X^4=\R^4$ we have more information about the structure of $P$: Using an isomorphism $\QR{SO(4)}{U(2)}\cong \C P^1 \cong \Szw$ the twistor space $P$ of $\R^4$ is (as a set) the trivial $\Szw-$fiber bundle over $\R^4$, i.e.\ $P= \R^4 \times \Szw$ (see \cite[Section~4]{Atiyah}). On the other hand $P$ carries a holomorphic structure which is not the standard holomorphic structure on $\C^2\times \Szw$ but a twisted one:  
 Let $H$ be the standard positive line bundle over $\C P^1$, then $P$ is isomorphic to $H \oplus H$ (the Whitney sum of $H$ with itself), see \cite[Section~4]{Atiyah}. This is a bundle over $\Szw$ with projection $p: H \oplus H \to \Szw$. Thus, we are in the following situation:
  \[
   \begin{xy}
    \xymatrix{
		& M \ar[dl]_{p\circ F} \ar[d]^F  \ar[dr]^f& \\
	\Szw & H\oplus H \cong P \ar[l]^/0.8em/p \ar[r]_/0.7em/\pi& \R^4
    }
   \end{xy}
  \]

\end{remark}

\begin{proposition}[\cite{Friedrich}]\label{prop:Twistor holomorphic}
Let $f: M \to X^4$ be an immersion of an oriented two dimensional manifold $M$ then the following conditions are equivalent
\begin{enumerate}
\item $f$ is twistor holomorphic;
\item the connection forms defined above satisfy on every neighborhood $U$ 
\[ w_2^4 + w_1^3 - \star w_2^3 + \star w_1^4 =0 \]
where $\star$ is the Hodge star operator with respect to the induced metric $f^*h$;
\item for all $x \in U$ 
\[ F(x)A^\circ_{11}(x) = A^\circ_{12}(x). \] \label{xi}
\end{enumerate}
\end{proposition}

\begin{proof}
Although this proposition corresponds to \cite[Proposition 2]{Friedrich} we give here a direct proof.
Fix a point $x$ and choose an orthonormal frame $\{E_1,E_2,N_1,N_2\}$ in a neighborhood $U$ as in Definition~\ref{lift}. 
As described in Definition~\ref{lift} the lift corresponds to a matrix $F(y) \in SO(4)$ for all $y \in U$. By definition of $F(y)$, being twistor holomorphic is a condition on the vertical part of $T_{f(y)}P$, i.e.
\begin{equation}\label{eq:holomorphic condition1}
 F(x)D_vF(x) = D_{I(v)}F(x) \text{ for all } v \in T_xM.\end{equation}
Observe that  conditions \eqref{eq:twisterconditions} imply $F(y)^2 = - \mathbf{1}$,  $F(y)^t = - F(y)$ (where $A^t$ denotes the transpose of $A$) and so $$DF(y)F(y) = - F(y)DF(y), \ \ DF(y)^t = - DF(y).$$ Therefore $DF(y)$ maps the tangent space $T_yM$ into the normal space $N_yM$. This can be seen as follows: 
$$\langle E_1, DF(y)E_2\rangle = \langle E_1, DF(y)F(y)E_1 \rangle = - \langle E_1, F(y)DF(y)  E_1\rangle = \langle E_2, DF(y)E_1\rangle.$$ But the antisymmetry of $DF(y)$ implies $$\langle E_1, DF(y)E_2\rangle = -\langle E_2, DF(y)E_1\rangle,$$ so $\langle E_1, DF(y)E_2\rangle=0$. Similarly one shows that $\langle N_1, DF(y)N_2\rangle=0$.
Furthermore we have 
\begin{align*}
DF(x)E_2&= DF(x)F(x)E_1 = - F(x)DF(x)E_1\\ 
\langle N_i, DF(x)E_j\rangle &= - \langle E_j, DF(x)N_i\rangle \text{ for } i,j=1,2.
\end{align*}
and conclude that \eqref{eq:holomorphic condition1} is satisfied if and only if
\[  F(x)D_vF(x)E_1 = D_{I(v)}F(x)E_1 \text{ for all } v \in T_xM.\]
To calculate $DF(x)E_1$ we differentiate $0 = \langle N_i, F(y)E_1\rangle$ along $v \in T_xM$ and obtain 
\begin{align*} 0 &=\langle N_i, DF(x)E_1\rangle + \langle D_vN_i, F(x)E_1\rangle - \langle F(x)N_i, D_vE_1\rangle \\
&= \langle N_i, DF(x)E_1\rangle - \langle N_i, D_vE_2\rangle - \sum_{j=1,2} \langle F(x)N_i,N_j\rangle \langle N_j, D_vE_1\rangle\\
&= \langle N_i, DF(x)E_1\rangle -\left( w_2^{i+2}(v) + \sum_{j=1,2} \langle F(x)N_i,N_j\rangle w_1^{j+2}(v)\right) 
\end{align*} 
using the 1-forms introduced in \eqref{eq:1-forms}. We calculate 
\[ D_{\cdot} F(x)E_1 = ( w_2^3 - w_1^4) N_1 + (w_2^4 + w_1^3) N_2, \]
showing the equivalence of (i) and (ii), since
\[  F(x)D_{\cdot}F(x)E_1 - D_{I(\cdot)}F(x)E_1 = (w_2^4 + w_1^3 - \star w_2^3 + \star w_1^4) N_1 + ( - w_2^3 + w_1^4 +\star w_2^4 + \star w_1^3) N_2. \]
It remains to check that (ii) is equivalent to (iii).  Evaluating (ii) in $E_1, E_2$, recalling $w_i^{k+2}(E_j)= \langle N_k, D_{E_j}E_i \rangle = \langle N_k, A_{ij}\rangle$ we have
\begin{align*}
\langle N_2, A_{12}\rangle + \langle N_1, A_{11}\rangle - \langle N_1, A_{22}\rangle + \langle N_2, A_{12}\rangle = 2 \left( \langle N_2, A^\circ_{12}\rangle + \langle N_1, A^\circ_{11}\rangle \right),\\
\langle N_2, A_{22}\rangle + \langle N_1, A_{12}\rangle + \langle N_1, A_{12}\rangle - \langle N_2, A_{11}\rangle = 2 \left( -\langle N_2, A^\circ_{11}\rangle + \langle N_1, A^\circ_{12}\rangle \right).
\end{align*}
This shows that (ii) holds if and only if $F(x) A^\circ_{11} = A^\circ_{12}$.
\end{proof}
%
%

\begin{corollary}\label{cor:eulernumber}
Let $M$ be an oriented  two dimensional  manifold and  $f:M\to \R^4$ an immersion into $\R^4$ then the following are equivalent 
\begin{enumerate}
\item $f$ is twistor holomorphic;
\item \[ \W(f) = 2\pi\left(\chi - e(\nu)\right)= 2\pi\left(\chi + |e(\nu)|\right),\]
i.e.\ equality holds in \eqref{Wintgenineq} and $e(\nu)\leq 0$. \label{gleichheit}
\end{enumerate}
\end{corollary}

\begin{proof}
The equivalence follows from the fact that \eqref{eq:Wintgenineq} is equivalent to condition (iii) in the previous proposition i.e.\ $F(x) A^\circ_{11} = A^\circ_{12}$, because in this case 
\begin{align*}
K^\perp &= 2 \langle A^\circ_{11}\wedge A^\circ_{12} ,N_1\wedge N_2 \rangle =2 \left(\langle A^\circ_{11}, N_1\rangle \langle A^\circ_{12}, N_2\rangle - \langle A^\circ_{11}, N_2\rangle\langle A^\circ_{12}, N_1\rangle\right)\\
&= - 2 \abs{A^\circ_{11}}^2 = - 2 \abs{A^\circ_{11}}\abs{A^\circ_{12}}.
\end{align*}
If  \eqref{eq:Wintgenineq} holds then either  $F(x) A^\circ_{11} = A^\circ_{12}$ or  $-F(x) A^\circ_{11} = A^\circ_{12}$ but since  $K^\perp$ must be non positive  such that equality holds, the second is excluded.
\end{proof}

\begin{remark}
 As the Veronese surface satisfies $\W(f)= 6\pi$ and $e(\nu)=-2$ (when the orientation of $\R^4$ is chosen appropriately) we get that the oriented double cover $\tilde f: \mathbb S^2\to\R^4$ is twistor holomorphic.
\end{remark}

%

Friedrich considered in \cite{Friedrich} twistor holomorphic immersions into $\R^4$ in detail. He used the special structure of $P$ to prove a kind of ``Weierstrass representation'' for such immersions.

\begin{theorem}[Friedrich \cite{Friedrich}] \label{weierstr}
 Let $M$ be an oriented two-dimensional manifold. Let $P\cong H\oplus H$ be the twistor space of $\R^4$ (see remark after Definition~\ref{lift}). 
 A holomorphic map $F:M\to P$ corresponds to a triple $(g, s_1, s_2)$, where 
 \begin{enumerate}
  \item $g: M\to \Szw$ is a meromorphic function,
  \item $s_1,s_2$ are holomorphic sections  of the bundle $g^\ast (H)$ over $M$.
 \end{enumerate}
Furthermore, there are holomorphic maps $\varphi^i,\psi^i: M_i\to\C$, where $i=1,2$ and  $M_1\defi\{g\neq\infty\}$ and $M_2\defi\{g\neq 0\}$ such that
\begin{align*}
s_1&= (\varphi^1,\varphi^2), \ s_2=(\psi^1,\psi^2) \ \ \text{ with }\\
\varphi^2 &= \frac{1}{g}\varphi^1 \ \text{ and } \ \psi^2 = \frac{1}{g}\psi^1 \ \text{ on } \ M_1\cap M_2.
\end{align*}
A holomorphic map $F:M\to P$ defines a twistor holomorphic immersion $f:M\to\R^4$ via $f=\pi\circ F$ if and only if 
\begin{align}\label{imm}
 |d s_1|+|d s_2|>0.
\end{align}
If (\ref{imm}) is satisfied, then $f$ is given by the formula
\begin{align}\label{weier}
 f=\left(\frac{\varphi^1 \bar g - \bar\psi^1}{1 + |g|^2} , \frac{\bar\psi^1 g + \varphi^1}{1 + |g|^2}\right).
\end{align}
Conversely, if $f$ is given by (\ref{weier}) with $s_1,s_2$ satisfying (\ref{imm}) then $f$ is a twistor holomorphic immersion. Any such twistor holomorphic immersion satisfies the formula
\begin{align}\label{Willm}
 \W(f) = 4 \pi \deg (g).
\end{align}
 
\end{theorem}

\begin{proof}
 The proof is done in Section~1, Remark~2 and Section~4, Example~4 of \cite{Friedrich}. We remark that the meromorphic function $g$ is defined by $g= p\circ F$.
\end{proof}

\begin{corollary}\label{cor:prop}
 Let $M$ be an oriented two-dimensional manifold and $f:M\to\R^4$ a twistor holomorphic immersion. Let $(g,s_1,s_2)$ be the triple corresponding to the lift $F:M\to P$ (cf. Theorem~\ref{weierstr}). Then the maps $\varphi^i$ and $\psi^i$, $i=1,2$, can be extended  to meromorphic functions $\varphi^i,\psi^i: M \to \Szw$, $i=1,2$, 
 with the following properties: Denote by $S_P(h)\defi \{x\in M: h(x)=\text{north pole} = \infty \}$ the poles and by $S_N(h)\defi\{x\in M: h(x)=\text{south pole}= 0\}$ the zeros of a meromorphic function $h:M\to \Szw$, and let $\ord_h(b)$, $b\in S_P(h)$ or $b\in S_N(h)$, be the order of the poles or zeros of $h$. Then we have that
 \begin{enumerate}
  \item $S_P(\varphi^1)\subset S_P(g) \ $ and $ \ \ord_{\varphi^1}(b)\leq \ord_g(b) \ \forall b \in S_P(\varphi^1)$, \label{h}
  \item $S_P(\varphi^2)\subset S_N(g) \ $ and $ \ \ord_{\varphi^2}(a)\leq \ord_g(a) \ \forall a \in S_P(\varphi^2)$,
  \item  $S_P(\psi^1)\subset S_P(g) \ $ and $ \ \ord_{\psi^1}(b)\leq \ord_g(b) \ \forall b \in S_P(\psi^1)$,
  \item  $S_P(\psi^2)\subset S_N(g) \ $ and $ \ \ord_{\psi^2}(a)\leq \ord_g(a) \ \forall a \in S_P(\psi^2)$.
 \end{enumerate}

\end{corollary}

\begin{proof}
 On $\C\setminus\left(S_P(g)\cup S_N(g)\right)$ we have $\varphi^2 g = \varphi^1$, and $\varphi^1:\C\setminus S_P(g) \to\C$, $\varphi^2:\C\setminus S_N(g)$ are holomorphic. Thus, either  $\lim_{z\to b}\varphi^1(z) =\infty$ for $b\in S_P(g)$ or $\varphi^1(b)$ is a zero of order bigger or equal to $-\ord_g(b)$. 
 In the latter case, $\varphi^1$ has a removable singularity in $b$ and can be extended smoothly. In the first case, $\varphi^1$ has a pole in $b$. 
 There are no other poles of $\varphi^1$, and the order of a pole of $\varphi^1$ cannot be bigger than that of $g$. Therefore, (\ref{h}) holds. The other three claims follow in the same way.
\end{proof}

\begin{proposition}\label{prop:involution}
The twistor space $P=\R^4\times \QR{SO(4)}{U(2)}$ of $\R^4$  naturally carries an antiholomorphic involution $J$ defined as being the identity on $\R^4$ and the multiplication by $-1$ on $\QR{SO(4)}{U(2)}$. The composition of this involution with a lift of an immersed surface into the twistor space gives the lift of the same surface with reversed orientation.
Furthermore, the involution is fiber preserving and induces the antiholomorphic involution $z \mapsto -\frac{1}{\overline{z}}$ on $\C P^1$.
\end{proposition}
\begin{proof}
As already mentioned the twistor space $P$ of $\R^4$ is isomorphic to $H\oplus H$ in the sense that the following diagram commutes:
  \[
   \begin{xy}
    \xymatrix{
		P \ar[r]^\psi \ar[dr]_{\pi} &  H \oplus H \ar[d]^{\tilde{\pi}}\\
		& \R^4 
    }
   \end{xy}
  \]
 \\
The projection $\tilde{\pi}$ is given by \eqref{weier}, compare Section 4, Example 4 of \cite{Friedrich}. One can understand $\tilde{\pi}$ as follows:
We define the local sections around a point $z \in \C P^1 \setminus \{[(0,1)], [(1,0)]\}$ with representative $(u_1,u_2) \in \C^2$ as
\[ \alpha(z)= \left( \frac{u_1}{u_2}, 1\right), \quad \beta(z) = \left(1, \frac{u_2}{u_1}\right).\]
A holomorphic section in $H \oplus H$ can be parametrized by the real 4-parameter family 
\begin{equation}\label{real4}
\zeta= (A \alpha(z) + B \beta(z), \bar{B} \alpha(z) - \bar{A} \beta(z))\,.
\end{equation}
The projection $\tilde{\pi}(\zeta)$ is then $(A,B) \in \C^2 = \R^4$.\\
The space $H\oplus H$ is holomorphically embedded in $\C^4$ by inclusion. We define the antiholomorphic involution
\[ \tilde{I}: \C^4 \to \C^4 \quad u=(u_1,u_2,u_3,u_4) \mapsto (\bar{u}_4, -\bar{u}_3, -\bar{u}_2, \bar{u}_1)\,.\]
Let $z \in \C P^1 \setminus\{[(0,1)]\}$ and $u \in H \oplus H$ with $p(u)=z$, i.e.~$z=\frac{u_1}{u_2}=\frac{u_3}{u_4}$, then $p\circ\tilde{I}(u)= -\frac{\bar{u}_2}{\bar{u}_1} $, hence $\tilde{I}$ defines an antiholomorphic involution on $\C P^1$  by $z \mapsto -\frac{1}{\bar{z}}$. Using the parametrization \eqref{real4} one readily checks that
\begin{equation}\label{identity on R^4}
\tilde{I}(\zeta)= \left(A \alpha\left(-\frac{1}{\bar{z}}\right) + B \beta\left(-\frac{1}{\bar{z}}\right), \bar{B} \alpha\left(-\frac{1}{\bar{z}}\right) - \bar{A} \beta\left(-\frac{1}{\bar{z}}\right) \right)\,.
\end{equation}
Due to the isomorphism $\psi: H\oplus H \to P$ we obtain an antiholomorphic involution on $P$ by $ J \defi \psi\circ \tilde{I} \circ \psi^{-1}$. Equation \eqref{identity on R^4} implies that $J$ is the identity on $\R^4$. It remains to show that $J$ corresponds to the multiplication by $-1$ on $ \QR{SO(4)}{U(2)}$. This can be seen as follows:
Let $f: M \to \R^4$ be a twistor holomorphic immersion with holomorphic lift $F: M \to P$, compare Definition~\ref{lift}. We denote by $\tilde{F}:=\psi\circ F$ the associated holomorphic map into $H\oplus H$ and let $\sigma: M \to M$ be an antiholomorphic involution on the Riemann surface $M$ reversing the orientation. Using $\tilde{I}$ we obtain a new holomorphic map $\tilde{F}_2: M \to H \oplus H$ by $\tilde{F}_2:=\tilde{I}\circ\tilde{F}\circ \sigma$. Furthermore, we have $\tilde{\pi}\circ \tilde{F}_2(p)=f(\sigma(p))=f(p)$ for all $p \in M$ due to \eqref{identity on R^4}. Hence $F_2:=\psi^{-1}\circ\tilde{F}_2: M \to P$ has to be the lift corresponding to the immersion $f\circ \sigma: M \to \R^4$ and therefore $J$ must be the multiplication by $-1$ on $\QR{SO(4)}{U(2)}$.
\end{proof}

%
%

\begin{corollary}\label{thmgcircI}
 Let $M$ be a two-dimensional oriented manifold and $f:M\to \R^4$ a twistor holomorphic immersion. Assume that $M$ is equipped with an antiholomorphic involution $I:M\to M$ without fixpoints such that $f\circ I = f$. 
 Let $F:M\to P$ be the lift into the twistor space and $(g,s_1,s_2)$ the corresponding triple (cf.~Theorem~\ref{weierstr}). Then we have that
 \begin{align}  \label{gcircI}
  g\circ I = -\frac{1}{\bar g}
 \end{align}
 and 
 \begin{align}\label{phicircI}
  \varphi^1 \circ I = \frac{\bar \psi^1}{\bar g}=\bar\psi^2 \ \ \text{ and }\ \ 
  \psi^1 \circ  I =  - \frac{\bar \varphi^1}{\bar g}= -\bar\varphi^2.
 \end{align}
The immersion $f$ is given by the formula
 \begin{align} \label{formf}
  f = (f_1,f_2)= \left(\frac{\varphi^1-\varphi^1\circ I}{g - g\circ I},\frac{\bar\psi^1-\bar \psi^1\circ I}{\bar g - \bar g\circ I}\right)=\left(\frac{ \varphi^1 - \bar \psi^2}{g + \frac{1}{\bar g}},\frac{\bar\psi^1 + \varphi^2}{\bar g + \frac{1}{g}} \right).
 \end{align}

\end{corollary}

\begin{proof}
Let $F: M\to P$ be the holomorphic lift of $f$ as in the statement above. Then $g= p\circ F$, where $p: P\to\mathbb C P^1$ is the projection in $H\oplus H\cong P$. 
Consider the holomorphic map $\tilde F\defi J\circ F\circ I:M\to P$, where $J:P\to P$ is the antiholomorphic involution from Proposition~\ref{prop:involution}. By assumption we have $f\circ I= f$ which implies (together with the properties of $J$) that $\tilde F$ is the lift of $f$, i.e.~$\tilde F =F$. As $J$ induces the antipodal map on $\C P^1$ we have that
\begin{align*}
 g\circ I = p\circ F\circ I = p\circ  J\circ F = - \frac{1}{\bar g},
\end{align*}
which is (\ref{gcircI}). For (\ref{phicircI}), we do the same argument but now on $H\oplus H$: Denote by $\psi: H\oplus H\to P$ the isomorphism as in the proof of Proposition~\ref{prop:involution}. The antiholomorphic involution $\tilde I$ on $H\oplus H$ from the same proposition has the property $ \tilde I \circ \tilde F\circ I=\tilde F,$
where $\tilde F\defi \psi\circ F$. By definition of $\tilde I$ we get $\bar\psi^2\circ I = \varphi^1$ and $-\bar\psi^1\circ I=\varphi^2$ which implies (\ref{phicircI}).
Formula \eqref{formf} is a consequence of \eqref{weier} and \eqref{phicircI}.
\end{proof}

\begin{lemma} \label{lemmaexg}
 Let $M=\QR{\C}{\Gamma}$ be a torus that carries an antiholomorphic involution $I: M\to M$ without fixpoints. Then $M$ is biholomorphically equivalent to a torus with a rectangular lattice. Moreover, there is a set of admissible parameters $\Lambda_0\neq\emptyset$ and a family of
 meromorphic functions $g_\lambda:M\to \Szw$, $\lambda \in\Lambda_0$, with $\deg (g_\lambda) = 4$ and $g_\lambda\circ I = -\frac{1}{\bar g_\lambda}$.
\end{lemma}

\begin{proof} 
By Proposition~\ref{biholom} we get a biholomorphic map $\varphi:M\to\tilde M$, where $\tilde M$ is generated by a ``canonical basis'' $(1,\tau)$. As $\tilde I\defi \varphi\circ I\circ\varphi^{-1}$ is an antiholomorphic involution without fixpoints on $\tilde M$ we know that $\tilde M$ is generated by a rectangular lattice and $\tilde I$ must be $\tilde I(z)=\bar z + \frac{1}{2}$ or $\tilde I(z)=-\bar z + \frac{\tau}{2}$, see Theorem~\ref{thminvolution}. If $\tilde I(z)=-\bar z + \frac{\tau}{2}$, then we go again to another lattice by a biholomorphic map $\psi: \tilde M \to \QR{\C}{\Gamma_1}$, $\psi(z) =\frac{z}{\tau}$. Then $\Gamma_1$ has generator $(\tau_1,1)$, $\tau_1=\frac{-\tau}{|\tau|^2}$ and 
\begin{align*}
 \psi\circ I\circ \psi^{-1}(z)= \psi (-\bar \tau \bar z + \frac{\tau}{2}) = -\frac{\bar \tau}{\tau} \bar z + \frac{1}{2}   =\bar z + \frac{1}{2}
\end{align*}
because $\tau$ is purely imaginary. Thus, we can assume that $I(z)=\bar z + \frac{1}{2}$ and we have a rectangular lattice generated by $(1,\tau)$.

The second step is the proof of the existence of $g$. As we are looking for an elliptic function of degree $4$, $g$ must have four poles $b_k$, $k=1,...,4$, and four zeros $a_k$, $k=1,...,4$ (counting with multiplicities). For the theory of elliptic functions see for example \cite{KoecherKrieg}. 
Such an elliptic function exists if and only if $\sum_{k=1}^4 b_k - \sum_{k=1}^4 a_k \in\Gamma$, see \cite[Section~1.6]{KoecherKrieg}. Consider the function $$h(z) \defi g(\bar z +\tfrac{1}{2}) \bar g(z).$$ We show that we can choose the poles and zeros of $g$ such that $h\equiv -1$, which is equivalent to (\ref{gcircI}). \\
As $g$ only has poles in $b_k$, we require $a_k = I(b_k)= \bar b_k + \frac{1}{2}$. Then $$\sum_{i=1}^4 b_k - \sum_{k=1}^4 ( \bar b_k + \frac{1}{2})= 2i\sum_{k=1}^4 \Im (b_k) - 2 \in\Gamma$$
is a necessary condition for the existence of such $g$. Thus, there must be an $m\in\Z$ with $\sum_{k=1}^4 \Im (b_k) = \frac{m}{2}\Im(\tau)$. As $I$ is an involution we have that $I(a_k) = b_k$, thus if $g\circ I$ has a pole in a point, then $\bar g$ has a zero of the same order at that point (and vice versa). It follows that $h$ has no poles. 
As $\bar h$ is elliptic without poles it is constant, and $h$ is constant as well. We have to find out if this constant can be $-1$. Define $\omega_0\defi \sum_{i=1}^4 b_k - \sum_{k=1}^4 ( \bar b_k + \frac{1}{2})= m\tau  - 2 \in\Gamma$. Then, up to a complex constant factor $c$, $g$ is of the form
\begin{align*}
 g(z) = e^{-\eta(\omega_0) z}\frac{\prod_{k=1}^4 \sigma (z-\bar b_k - \frac{1}{2})}{\prod_{k=1}^4\sigma (z-b_k)},
\end{align*}
where $\sigma:\C\to\C$ denotes the Weierstrass Sigma Function and $\eta:\Gamma\to\C$ is the group homomorphism that satisfies the Legendre relation, i.e.
\begin{align}\label{legendre}
 \eta(\omega_2)\,\omega_1 - \eta(\omega_1)\,\omega_2 = 2\pi i, \ \text{ if } \  \Im\left(\frac{\omega_1}{\omega_2}\right)>0.
\end{align}
We collect some facts about $\sigma$ and $\eta$, see \cite[Section~1.6]{KoecherKrieg}: The function $\sigma$ is an entire function that has exactly in all lattice points zeros of order one. As it is nonconstant and has no poles it cannot be doubly periodic. 
But it has the property
\begin{align*}
 \sigma(z + \omega) = - e^{\eta(\omega)\left(z + \frac{\omega}{2}\right)} \sigma (z),
\end{align*}
when $\frac{\omega}{2}\not\in\Gamma$. If the lattice is real, then $\bar \sigma (z) =\sigma (\bar z)$. This can be seen in the representation formula
\begin{align*}
 \sigma (z) = z \prod_{0\neq \omega\in\Gamma}\left(1-\frac{z}{\omega}\right) e^{\frac{z}{\omega}+ \frac{1}{2}\left(\frac{z}{\omega}\right)^2}.
\end{align*}
For a rectangular lattice, $\eta$ has the property that 
\begin{align}\label{eta}\begin{split}
 \eta (\omega) \in i\R \ \text{ for } \ \omega \in\Gamma\cap i\R, \\
 \eta (\tilde\omega)\in \R \ \text{ for } \tilde\omega\in \Gamma\cap\R. \end{split}
\end{align}
We use these properties to get
\begin{align}\label{h1}
 h(z) & = \exp\left(-\eta(\omega_0) (\bar z + \frac{1}{2}) - \bar\eta (\omega_0) \bar z\right) \frac{\prod_{k=1}^4 \sigma (\bar z - \bar b_k)}{\prod_{k=1}^4 \bar \sigma ( z -  b_k)} \frac{\prod_{k=1}^4\bar \sigma ( z - \bar b_k - \frac{1}{2})}{\prod_{k=1}^4 \sigma (\bar z -  b_k + \frac{1}{2})}\\
 & = \exp\left(- 2 \Re(\eta(\omega_0)) \bar z - \frac{\eta(\omega_0)}{2}\right)\frac{\prod_{k=1}^4\bar \sigma ( z - \bar b_k - \frac{1}{2})}{\prod_{k=1}^4 \sigma (\bar z -  b_k + \frac{1}{2})}\nonumber\\
& =  \exp\left(- 2 \Re(\eta(\omega_0)) \bar z - \frac{\eta(\omega_0)}{2}\right) \frac{\prod_{k=1}^4\bar \sigma ( z - \bar b_k - \frac{1}{2})}{\prod_{k=1}^4 \sigma (\bar z -  b_k - \frac{1}{2})} (-1)^4 \exp\left( -\eta(1) \sum_{k=1}^4 (\bar z - b_k)\right)\nonumber\\
&= \exp\left(- 2 \Re(\eta(\omega_0)) \bar z - \frac{\eta(\omega_0)}{2} - 4 \eta(1)\bar z + \eta(1)\sum_{k=1}^4 b_k \right).  \nonumber
\end{align}
As $\eta(\omega + \tilde \omega) =\eta(\omega) + \eta(\tilde\omega)$ $\forall \omega,\tilde\omega\in\Gamma$ ($\eta$ is a group homomorphism) and $\eta(0)=0$ we get that 
\begin{align*}
 \eta(\omega_0)= m\eta(\tau) - 2\eta(1) \ \text{ and }\ \eta(\tau)\in i\R \ \text{ and }\ \eta(1) \in\R.
\end{align*}
Thus, (\ref{h1}) yields
\begin{align*}
 h(z) &= \exp\left( + 4\eta(1) \bar z + \eta(1) - \frac{m}{2}\eta(\tau) - 4 \eta(1)\bar z + \eta(1) \sum_{k=1}^4 \Re(b_k) + \eta (1) \frac{m}{2}\tau\right)\\
 &= \exp\left(\eta(1)\left( 1 + \sum_{k=1}^4 \Re(b_k)\right)\right) 
 \exp\left( m \pi i  \right),
\end{align*}
where we used the Legendre relation (\ref{legendre}) and property (\ref{eta}) in the last step. Thus, for every combination of poles $b_k$, $k=1,...,4$ that satisfies $i\sum_{k=1}^4 \Im (b_k) = \frac{2l + 1}{2} \tau$ for an $l\in \Z$ we define $R\defi\sum_{k=1}^4\Re(b_k)$ and choose $c\defi e^{-\frac{\eta(1)}{2}\left( 1 + R\right)}$. 
Then we have with $\tilde g (z) \defi c g(z)$ ($g$ as above) that 
\begin{align*}
 h(z) = |c|^2 e^{\eta(1)\left( 1 + R\right)}\cdot (-1)=-1,
\end{align*}
which is equivalent to (\ref{gcircI}). 
As we can assume that $b_k\in [0,1]\times [0,\Im(\tau)]$ it suffices to consider poles such that $i\sum_{k=1}^4 \Im (b_k) = \frac{2l + 1}{2} \tau$ with $l\in\{0,1,2,3\}$. We collect all such possible $b_k's$ in $\Lambda_0$. The set $\Lambda_0$ is obviously not empty.

\end{proof}

\begin{proposition}\label{prop:immersion}
 Let $M=\QR{\C}{\Gamma}$ be a torus that carries an antiholomorphic involution $I:M\to M$  without fixpoints. Let $g:M\to\Szw$ be meromorphic with $g\circ I=-\frac{1}{\bar g}$ (coming from Lemma~\ref{lemmaexg}). If there is a meromorphic function $\varphi^1: M\to\Szw$ with $2\leq \deg(\varphi^1)\leq 4$, $S_P(\varphi^1)\subset S_P(g)$, $\ord_{\varphi^1}(b)\leq \ord_g(b)$ $\forall b\in S_P(\varphi^1)$ and $\varphi^1\neq c g + \tilde c$ $\forall c,\tilde c\in \C$ then there are unique meromorphic functions $\psi^1,\psi^2,\varphi^2:M\to \Szw$ such that the triple $(g,s_1,s_2)$ ($s_1=(\varphi^1,\varphi^2)$, $s_2=(\psi^1,\psi^2)$) corresponds to a twistor holomorphic immersion $f:M\to\R^4$. 
 In particular, the properties of Theorem~\ref{weierstr} and Corollary~\ref{cor:prop} are satisfied.

\end{proposition}
\begin{proof}
 As the existence of $\varphi^1$ is assumed we define $\varphi^2 = \frac{\varphi^1}{g}$ and $\psi^2 =\bar\varphi^1\circ I$  and $\psi^1 = - \bar \varphi^2 \circ I$. This defines $\psi^1,\psi^2,\varphi^2:M\to \Szw$ uniquely and we have all the properties of Corollary~\ref{cor:prop}. Then we define
 \begin{align*}
   f\defi\left(\frac{\varphi^1 \bar g - \bar\psi^1}{1 + |g|^2} , \frac{\bar\psi^1 g + \varphi^1}{1 + |g|^2}\right)=\left(\frac{\varphi^1-\varphi^1\circ I}{g - g\circ I},\frac{\bar\psi^1-\bar \psi^1\circ I}{\bar g - \bar g\circ I}\right),
 \end{align*}
which is formula (\ref{weier}) and (\ref{formf}). In this way, we also know $f\circ I = f$.
We claim that $f$ is not constant. If $f_1=c$ for a constant $c\in\C$ then $\varphi^1-\varphi^1\circ I = c(g-g\circ I)$. This is equivalent to $\varphi^1-cg = (\varphi^1-cg)\circ I$, which implies that $\varphi^1-cg$ is holomorphic (as map into $\Szw$) and antiholomorphic. Thus, it must be a constant. But this contradicts $\varphi^1\neq c g + \tilde c$ $\forall c,\tilde c\in C$.\\
We do not know yet if $\abs{d\varphi^1} + \abs{d\varphi^2} + \abs{d\psi^1} + \abs{d\psi^2}>0$. This is necessary for $f$ to be an immersion, see Theorem~\ref{weierstr}. Define $$B\defi\{z\in M: \abs{d\varphi^1} (z) + \abs{d\varphi^2}(z) + \abs{d\psi^1}(z) + \abs{d\psi^2}(z)=0 \}.$$ 
We assume $B\neq \emptyset$.
Considering $\varphi^i,\psi^i$ as elliptic functions with finite degree we know that $\abs{B}<\infty$. As $I$ has no fixpoints $|B|$ is an even number.
By Friedrich's construction, $f: M\to\R^4$ is a branched conformal immersion with branch points in $B$. 
The Riemann Hurwitz formula for covering maps with ramification points yields the formula
\begin{align*}
 \W (f) = 4\pi \deg(g)=16 \pi
\end{align*}
as shown by Friedrich, see \cite[Section~4, Example~4]{Friedrich}. We combine this with the Gauss Bonnet formula for conformal branched immersions \cite[Theorem~4]{Eschenburg},
\begin{align*}
 \int_M K = 2\pi \left(\chi(M) + \sum_{p\in B} m(p)\right),
\end{align*}
where $m(p)$ is the branching order in $p$, to get 
\begin{align*}
 \frac{1}{4}\int_M|A|^2= \W(f) -\frac{1}{2}\int_M K<\infty.
\end{align*}
Since $M$ is compact we have that $\operatorname{Vol}(f(M))<\infty$. Thus, $f:M\to\R^4$ is a $W^{2,2}-$conformal branched immersion and we can apply \cite{KuwertLi}. For that, fix any $x_0= f(p)$ for some $p\in B$. Then $\sum_{p\in f^{-1}(x_0)}(m(p)+1)\geq 4$. Define $\hat f\defi S\circ J\circ f$, where $J(x)\defi x_ 0 + \frac{x-x_0}{\abs{x-x_0}^2}$ and $S:\R^4\to\R^4$ is any reflection. 
Then $\hat f: M\setminus B\to\R^4$ is twistor holomorphic because $A^\circ$ does not change by a conformal transformation and by Proposition~\ref{prop:Twistor holomorphic} (\ref{xi}) (note that the reflection makes sure that $S\circ J$ is orientation preserving).
We apply (3.1) of \cite{KuwertLi}
and get
\begin{align} \label{Kuw}
 \W(\hat f) = \W(f) - 4\pi \sum_{p\in f^{-1}(x_0)}(m(p)+1) \leq 0.
\end{align}
Hence, $\hat f:M\setminus B\to\R^4$ is superminimal (i.e.\ twistor holomorphic and minimal). By a classical result of Eisenhart \cite{Eisenhart} $\hat f$ is locally given by two (anti-)holomorphic function $\hat f=(h_1,h_2)$. But this yields a contradiction because $\hat f  \circ I = \hat f$ implies that the components of $\hat f$ are holomorphic and antiholomorphic and hence constant. 
\end{proof}

We now restate and prove our main theorem.

\begin{theorem} \label{Unserthm}
%
On each torus $M_r$ with rectangular lattice generated by $(1,ir)$, $r\in\R^+$, there is a set of admissible parameters $\Lambda\neq \emptyset$ such that there are smooth conformal immersions $\hat f^r_\lambda: M_r\to\R^4$ that are twistor holomorphic and double covers of Klein bottles. The corresponding immersed Klein bottles $f^r_\lambda:K\to\R^4$, $\lambda\in\Lambda$, satisfy $\W(f^r_\lambda)= 8\pi$ and $e(\nu^r_\lambda)= -4$. By reversing the orientation of $\R^4$ we get a family of immersions $\tilde f^r_\lambda$ with $W(\tilde f^r_\lambda)= 8\pi$ and $e(\tilde\nu^r_\lambda)= +4$. \\
Every immersion $f:K\to\R^4$ with $\W(f)= 8\pi$ and $e(\nu)\in\{-4, +4\}$ is an embedding and is either an element of $\{f^r_\lambda:\lambda \in\Lambda\}$ or of $\{\tilde f^r_\lambda:\lambda \in\Lambda\}$. 
 Furthermore, every such immersion is a minimizer of the Willmore energy in its regular homotopy class, thus it is a Willmore surface.
\end{theorem}

\begin{proof}
Every Klein bottle $N$ is the quotient of its oriented double cover $q: M\to N$ and the group $\{id, I\} = \langle I\rangle$, where $I: M\to M$ is the antiholomorphic order two deck transformation on a torus $M$, i.e.\ $N=\QR{M}{\langle I \rangle}$ and $q: M\to\QR{M}{\langle I \rangle}$. \\
We consider the rectangular lattice generated by $(1,\tau)$ with  $\Im(\tau)>0$ and the involution $I(z)=\bar z + \frac{1}{2}$. 
As the imaginary part of $\tau$ is not fixed we get the parameter $r\defi \Im(\tau)$ and a family of tori $\{M_r\}_{r\in\R^+}$ with the involution $I$. 
From now on we keep $r$ fixed and denote $N\defi \QR{M_r}{\langle I \rangle}$.
We choose $b_k \in [0,1]\times [0,\Im(\tau)]=: F$, $k=1,...,4$, with $i\sum_{k=1,...,4}\Im(b_k) = \frac{m}{2}\tau$ with $m$ odd. As in the proof of Lemma~\ref{lemmaexg}, each of such a combination yields a meromorphic functions $g$ with (\ref{gcircI}). Consider any $p_1\in F \setminus \{b_1,b_2\}$ and define $p_2\defi b_1 + b_2 - p_1$. Then 
\begin{align}\label{vc}
 b_1 + b_2 - p_1 - p_2 = 0\in\Gamma
\end{align}
If $p_2\in\{b_1 + \Gamma ,b_2 + \Gamma\}$ then we go to $(\tilde p_1, \tilde p_2) = (p_1 + \epsilon, p_2 - \epsilon)$ such that (\ref{vc}) is still satisfied and $\{\tilde p_1,\tilde p_2\}\cap \left(\{b_1 + \Gamma\}\cup\{b_2 + \Gamma\}\right)=\emptyset$. By the existence theorem for elliptic functions there exists a meromorphic $\varphi^1: M \to \Szw$ with poles in $b_1,b_2$ and zeros in $\tilde p_1, \tilde p_2$. By construction and with $\deg(\varphi^1)= 2\neq \deg(g)$ we have proven the existence of a $\varphi^1$ that we need for Proposition~\ref{prop:immersion}.\footnote{The existence of $\varphi^1$ seems to be clear, but there are indeed cases, where we have to be careful. 
If $b_1=..=b_4=:b$ are the poles of $g$ then we cannot construct $\varphi$ with double pole in $b$ and a double zero in $I(b)$, because then we have that $i\Im(b)=\frac{m}{8}\tau$ with $m$ odd and $i \Im (b) =\frac{k}{4}\tau$ with $k\in\Z$, a contradiction.} Of course, this $\varphi_1$ is only one choice. In $\Lambda$ we include all possible $g's$ and $\varphi_1's$.
Then, $$ \tilde f\defi\left(\frac{\varphi^1 \bar g - \bar\psi^1}{1 + |g|^2} , \frac{\bar\psi^1 g + \varphi^1}{1 + |g|^2}\right)$$ is a twistor holomorphic immersion with $\W(\tilde f) = 16 \pi$ and 
\begin{align*}
 e(\nu) = \chi(M) - \frac{1}{2\pi}\W(\tilde f) = -8
\end{align*}
due to Corollary~\ref{cor:eulernumber} (\ref{gleichheit}). We define $f$ by $\tilde f = f \circ q$ and get immersions $f: N \to\R^4$ with $\W(f)= 8\pi$ and $e(\nu_f)=-4$ (equality $e(\nu_{\tilde f})=2 e(\nu_f)$ can be seen for example in (\ref{hf})). By reversing the orientation of $\R^4$ and repeating the construction of $\tilde f$ and $f$ we get immersions $\hat f: N\to\R^4$ with $\W(\hat f) = 8\pi$ and $e(\nu_{\hat f})= +4$. Note that in this case $\tilde{ \hat f} : M \to\R^2$ is not twistor holomorphic (Corollary~\ref{cor:eulernumber}).\\
On the other hand, every immersion $f:N\to\R^4$ with $\W(f)=8\pi$ and $e(\nu)\in\{+4,-4\}$ has all the properties shown in Theorem~\ref{weierstr}, Corollary~\ref{cor:prop}, Corollary~\ref{thmgcircI}. The proof of Lemma~\ref{lemmaexg} shows that every $g$ from the triple $(g,s_1,s_2)$ must be one of the $g_\lambda$ that we found for our surfaces. Also $\varphi_1$ must be one of ours. Thus, by the ``Weierstrass representation'' of Friedrich $f$  must be in $\{f_\lambda:\lambda\in\Lambda\}$ or $\{\hat f_\lambda:\lambda\in\Lambda\}$.\\ 
It remains to check that every immersion $f:N\to\R^4$ with $\W(f)=8\pi$ and $e(\nu)\in\{+4,-4\}$ is an embedding.
We repeat an argument from the proof of Proposition~\ref{prop:immersion}. If $e(\nu) =+4$, then we reverse the orientation of $\R^4$ and get an immersion with $e(\nu) = -4$. We go to the oriented double cover and get an immersion $\tilde f: M \to\R^4$ with $\W(\tilde f)=16\pi$ and $e(\nu)= -8$. 
As equality is satisfied in the Wintgen inequality $\tilde f$ is twistor holomorphic (Corollary~\ref{cor:eulernumber}). If $ f$ has a double point, then $\tilde f$ has a quadruple point $x_0\in\R^4$. Inverting at $\partial B_1(x_0) $ and reflecting in $\R^4$ yields 
$\W(\hat f)=0$ as in (\ref{Kuw}), where $\hat f\defi S\circ J\circ\tilde f$ ($J$ is the inversion, $S$ the reflection). As $S\circ J$ is conformal and orientation preserving $\hat f$ is still twistor holomorphic (Proposition~\ref{prop:Twistor holomorphic} (\ref{xi})). But every superminimal immersion into $\R^4$ is locally given by two (anti-)holomorphic functions. As $\hat f \circ I = \hat f$ for $I$ antiholomorphic, $\hat f_1$ and $\hat f_2$ must be constant. Thus, it cannot be an immersion, a contradiction.\\
Corollary~\ref{cor:Wgeq8pi} shows that every immersion with the properties as above is a minimizer in its regular homotopy class. As Willmore surfaces are defined as critical points of the Willmore energy under compactly supported variations the discovered immersions are Willmore surfaces.
\end{proof}

\begin{corollary}
 Let $K$ be a Klein bottle and $f:K\to\R^4$ an immersion with $\W(f)=8\pi$ and $\abs{e(\nu)}=4$. Let $q:M\to K$ be the oriented double cover on the corresponding torus $M$. Then $M$ is biholomorphically equivalent to a torus $M_r$ generated by $(1,ir)$, $r\in\R^+$, via a map $\varphi:M\to M_r$. Moreover, there are $f^r_\lambda,\tilde f^r_\lambda, q_r$ coming from Theorem~\ref{Unserthm} such that $f\circ q=f^r_\lambda\circ q_r\circ \varphi$ for $e(\nu)=- 4$ or $f\circ q=\tilde f^r_\lambda\circ q_r\circ \varphi$ for $e(\nu)=+ 4$. Thus, the surface $f(K)$ is one of the surfaces obtained in Theorem~\ref{Unserthm}.
\end{corollary}

\begin{proof}
  By possibly changing the orientation of $\R^4$ we can assume $e(\nu)=-4$. By Corollary~\ref{cor:eulernumber} we know that $f$ is twistor holomorphic. The work of Friedrich \cite{Friedrich} yields the existence of a triple $(g,s_1,s_2)$ as in Theorem~\ref{weierstr}. By Corollary~\ref{thmgcircI} we have $g\circ I =-\frac{1}{\bar g}$, where $I$ comes from the order two deck transformation of the oriented cover. Lemma~\ref{lemmaexg} shows that there is a biholomorphic map $\varphi:M \to M_r$ where $M_r$ is generated by $(1,ir)$ for an $r\in\R^+$. From Theorem~\ref{Unserthm} we get that $f\circ q$ must be $f^r_\lambda\circ q_r\circ \varphi$ for a $\lambda \in\Lambda$.
\end{proof}

\begin{proposition}\label{prop: SO(4) action}
The Lie group $SO(4)$ acts naturally and fiber preserving on the twistor space $P$. It induces a fiber preserving action on $H\oplus H$. The induced action on $\C P^1$ is the action of the 3-dimensional Lie subgroup $G$ of the M\"obius group on $\C P^1$ that commutes with the antipodal map $z \mapsto -\frac{1}{\bar{z}}$:
\[ G:= \left\{ m(z)=\frac{az +b}{\bar{b}z-\bar{a}} \colon a,b \in \C, \abs{a}^2 + \abs{b}^2 =1\right\}.\]
\end{proposition}

\begin{proof}
\emph{Claim 1:} 
The $SO(4)$ action $\varphi: SO(4) \times P \to P$ defined as
\[ O\cdot (y,j):= (Oy, OjO^t) \]
is natural and fiber preserving. \\
\emph{Proof of Claim 1:} The action preserves by definition the fibers. It is natural in the sense that if  $f: M \to \R^4$ is a given immersion with corresponding lift $F: M \to \R^4$ then for any $O \in SO(4)$ the map $O\cdot F$ is the lift of the immersion $Of$. Fix a point $x \in M$ and an orthonormal frame $\{ E_1, E_2, N_1, N_2 \}$ in a neighbourhood $U$ of $x$ as in Definition~\ref{lift} with related matrix $F(y) \in SO(4)$, $y\in U$. As $F$ satisfies conditions $i)$ to $ iv)$ so does $OF(y)O^t$. Furthermore $\{ OE_1, OE_2, ON_1, ON_2 \}$ is an orthonormal frame of $Of$ and $OF(y)O^t$ satisfies obviously the conditions \eqref{eq:twisterconditions} with respect to the new orthonormal frame.\\
\emph{Claim 2:} The action is holomorphic.\\
\emph{Proof of Claim 2:}
Given a point $(p,j) \in P$ the complex structure $\mathcal{J}_{(p,j)}$ on 
\[T_{(p,j)}P=T^H_{(p,j)}P \oplus T^V_{(p,j)}P = T_p\R^4 \oplus
 T_j\QR{SO(4)}{U(2)}\]
is by definition the multiplication by $j$ on $T_p\R^4$ and by $j$ on $T_j\QR{SO(4)}{U(2)}$. Hence given $X+Y \in T_p\R^4 \oplus
 T_j\QR{SO(4)}{U(2)}$ and $O \in SO(4)$ we have
\begin{align*}
dO\mathcal{J}_{(p,j)}\left(X+Y\right) &= dO \left(jX + jY\right)\\
&= OjO^t \, OX + OjO^t \, OYO^t = \mathcal{J}_{O\cdot (p,j)} dO(X+Y).
\end{align*}
Thus, the action is holomorphic.\\
\emph{Claim 3:} $SO(4)$ acts naturally, holomorphically and fiber preserving on $H \oplus H$. Hence it induces a group homomorphism $h: SO(4) \mapsto \operatorname{Aut}(\hat{\C})$, where $\operatorname{Aut}(\hat{\C})$ is the M\"obius group of the Riemann sphere.\\
\emph{Proof of Claim 3:} The isomorphism $\psi: P=\R^4 \times \QR{SO(4)}{U(2)} \to H\oplus H$ induces a natural, holomorphic action of $SO(4)$ on $H\oplus H$ by composition:
\[ O\cdot (u,v)= \psi\circ O \cdot \circ \psi^{-1}(u,v).\]
Recall that parallel transport (translation in $\R^4$) defines a fibration of $P$ over one of its fibers, compare Remark 2 \cite{Friedrich}. This fibration defines the isomorphism $\psi$. Hence we have a commutative diagram (compare the remark below):
  \[
   \begin{xy}
    \xymatrix{ \QR{SO(4)}{U(2)}  \ar[r]_\phi & \C P^1\\
    P  \ar[u]  \ar[r]_\psi & H \oplus H \ar[u]_p}
   \end{xy}
  \]
The action of $SO(4)$ on the $\QR{SO(4)}{U(2)}$-factor of $P$ is independent of the basepoint in $\R^4$ and therefore the induced action on $H\oplus H$ is fiber preserving. Therefore, the $SO(4)$ action on $P$ induces an action of $SO(4)$ on $\QR{SO(4)}{U(2)}$ and via the isomorphism $\phi: \QR{SO(4)}{U(2)} \to \C P^1$ it induces also an action on $\C P^1$. The action is holomorphic as proven in Claim~2. Therefore, $SO(4)$ acts on $\C P^1$ as biholomorphic maps. The holomorphic automorphism group of the Riemann sphere $\C P^1\cong \hat{\C}$ is the M\"obius group i.e.\ all rational functions of the form
\[  m(z)=\frac{az +b}{cz+d} \text{ with } a,b,c,d \in \C,\  ad-bc \neq 0.\]
\emph{Claim 4:} Let $G$ be the image of the group homomorphism $h: SO(4) \to \operatorname{Aut}(\hat{\C})$ then 
\[G = \left\{ m \in \operatorname{Aut}(\hat{\C}) \colon \frac{-1}{\overline{m(z)}} = m\left(\frac{-1}{\bar{z}}\right)\right\}; \]
i.e. $G$ is a 3-dimensional Lie-subgroup of $\operatorname{Aut}(\hat{\C})$.\\
\emph{Proof of Claim 4:} $h$ is induced by the group homomorphism $\phi$. Therefore, its kernel corresponds to the normal subgroup
\[ N:=\{ O \in SO(4) \colon  O\cdot j = OjO^t=j \quad \forall j \in {\QR{SO(4)}{U(2)}} \}. \]
The group $N$ can be determined explicitly using the isomorphism $\operatorname{Sp}(1)\otimes\operatorname{Sp}(1) \to SO(4)$ defined by \[(a,b)\cdot q =  a q \bar{b} \text{ for } q \in \Quat\cong\R^4 \text{ and } (a,b) \in \operatorname{Sp}(1)\otimes \operatorname{Sp}(1),\]
where $\operatorname{Sp}(1)$ is the group of unit quaternions, compare \cite[Proposition 1.1]{Salamon}. Conditions $i)$ to $ iv)$ in Definition~\ref{twistorspace} determine
\[ \QR{SO(4)}{U(2)} \cong \{ (1, c) \colon \bar{c}=-c \in \operatorname{Sp}(1) \}. \]
Hence $(a,b) \in N$ iff $(a,b)(1,c)(a,b)^t = (1,c)$ for all $c \in \Quat$, $\abs{c}=1, \bar{c}=-c$. This simplifies to
\begin{align*}
(a,b)(1,c)(a,b)^t &= (a,b)(1,c)(\bar{a},\bar{b})= (1, b c\bar{b}) = (1, c)\quad \forall c \in \operatorname{Sp}(1), \bar{c}= -c\\
 \Leftrightarrow \quad cb &= bc \quad \forall c \in \operatorname{Sp}(1), \bar{c}= -c.
\end{align*}
The last line implies that $b$ has to be real and since $\abs{b}=1$ we conclude $b\in\{-1, +1\}$. Furthermore, we have
\[ N=\{ (a,1) \colon a \in \operatorname{Sp}(1) \}  \]
which is a 3-dimensional Lie subgroup of $SO(4)$. The Lie group $\QR{SO(4)}{N}$ is as well 3-dimensional. $G$ is isomorphic to $\QR{SO(4)}{N}$ by the first isomorphism theorem.\\
Recall the antiholomorphic involution $J$ on the twistor space $P$ and the corresponding involution $\tilde{I}$ on $H\oplus H$ introduced in Proposition~\ref{prop:involution}. Applying $J$  corresponds to reversing the orientation of an immersed surface $f: M \to \R^4$. Since reversing the orientation of the manifold $M$ commutes with the $SO(4)$ action on $\R^4$ the natural associated maps on the whole space and the base have to  commute as well i.e.
\begin{align}\label{eq:commuting} O\cdot \tilde{I}(u,v) &= \tilde{I}( O\cdot(u,v))  &&\forall (u,v) \in H\oplus H, O \in SO(4) \ \ \text{ and}\nonumber\\
\frac{-1}{\overline{m(z)}} &= m\left( \frac{-1}{\bar{z}} \right) &&\forall z \in \C, m \in G
\end{align}
by the properties of $J$.
The subgroup $H$ of $\operatorname{Aut}(\hat{\C})$ that commutes with the antipodal map $z \mapsto - \frac{1}{\bar{z}}$ is readily calculated to be 
\[ H=\{m \in \operatorname{Aut}(\hat{\C}) \colon m(z)=\frac{az +b}{\bar{b}z-\bar{a}} \text{ with } a,b \in \C, \abs{a}^2 + \abs{b}^2 =1\}.\]
We observe that $H$ is as well a 3-dimensional connected Lie subgroup. Property~\eqref{eq:commuting} implies that $G \subset H$. Since $h: \QR{SO(4)}{N} \to G$ is a Lie group isomorphism $G$ is open and closed. As observed before $G$ and $H$ are both connected and of dimension 3. Hence we finally conclude that $H=G$.
\end{proof}
\begin{remark}
The translation invariance of the isomorphism $\psi: P \to H\oplus H$ can also be seen  in the formulas of Friedrich as follows: Let $z \in \C P^1 \mapsto (u,v) \in H \oplus H$ be holomorphic sections with $u(z_0)=A\alpha(z_0) + B\beta(z_0)$, $ v(z_0)= \bar{B} \alpha(z_0) - \bar{A} \beta(z_0)$. Consider the holomorphic sections $\tilde{u}(z)= u(z)- \left(A\alpha(z) + B\beta(z)\right), \tilde{v}(z)= v(z) - \left(\bar{B} \alpha(z) - \bar{A} \beta(z)\right)$. These sections still satisfy $p(u)=z=p(\tilde{u})$ and $p(v)= z = p(\tilde{v})$, but $\tilde{\pi}(u,v)= \tilde{\pi}(\tilde{u}, \tilde{v}) - (A,B)$.\\
The isomorphism $\varphi: \C P^1 \to \QR{SO(4)}{U(2)}$ can  explicitly be stated identifying $\C^2$ with the quaternions $\Quat$. Fix $g \in \C\cup \{\infty\}$ and let $\gamma \in \Quat$ be the  unit-quaternion $\gamma=\left(\frac{-1}{\sqrt{1 + \abs{g}^2}}, \frac{g}{\sqrt{1+ \abs{g}^2}}\right)$. The map $\varphi$ can now be stated using the quaternionic multiplication to be
\[ g \in \C P^1 \mapsto A_g \in \QR{SO(4)}{U(2)} \quad A_gq = q  ( - \bar{\gamma}  i  \gamma) \quad \forall q \in \Quat \cong \R^4,\]
which is equivalent to
$ g\mapsto \frac{1}{\abs{g}^2 + 1}
\begin{pmatrix}
0 & -\abs{g}^2 + 1 & 2g_2 & -2g_1 \\
\abs{g}^2 - 1 & 0 &  2g_1 & 2 g_2 \\
-2 g_2 & -2g_1 & 0 & \abs{g}^2-1 \\
2 g_1 & -2 g_2 & -\abs{g}^2 + 1 & 0
\end{pmatrix}.
$  
\end{remark}

\begin{corollary}\label{cor:uniqueness of the conformal class}
Consider $f^{r_i}_i: K\to \R^4$, $i=1,2$, a pair of Klein bottles with $\W(f_i^{r_i})=8\pi$ and $\abs{e(\nu_i^{r_i})}=4$. Let  $\Phi: \R^4 \to \R^4$ be a conformal diffeomorphism such that $f_1^{r_1}(K)=\Phi\circ f_2^{r_2}(K)$ then $r_1=r_2$.
\end{corollary}
\begin{proof}
We use the notation $f_i\defi f^{r_i}_i$ and keep in mind that the double covers possibly live on different lattices. After changing the orientation of $\R^4$ we may assume w.l.o.g that $e(\nu_1) = -4$. The Willmore energy and the Euler normal number are conformally invariant hence $f_2':=\Phi\circ f_2: K \to \R^4$ is a Klein bottle with $\W(f'_2)= 8\pi$ and $\abs{e(\nu'_2)}=4$. The Euler normal number of an immersion only depends on the image and not on the particular chosen immersion. Since $f_1(K)=f'_2(K)$ we deduce $e(\nu_1)=e(\nu_2')=-4$. Hence it is sufficient to prove the statement under the assumption that $\Phi$ is the identity and $e(\nu_i)=-4, i=1,2$.\\
Let $q_i: M_{r_i}=\QR{\C}{\Gamma_i} \to K$ be the oriented double cover of the related tori. By Theorem \eqref{Unserthm} $f_i\circ q_i: M_{r_i} \to \R^4$ are twistor holomorphic with related holomorphic lifts $F_i: M_{r_i} \to P \cong H \oplus H$. By assumption we have $F_1(M_{r_1})=F_2(M_{r_2})$. We also have that $f_i\circ q_i \circ I(z) = f_i\circ q_i(z)$ $\forall z \in \C$, where $I(z)=\bar{z}+ \frac12$.  The maps $F_i$ restricted to one fundamental domain are homeomorphism onto their image. Hence $G=F_1^{-1}\circ F_2: M_{r_2} \to M_{r_1}$ is a homeomorphism between tori. As $F_1, F_2$ are conformal and orientation preserving, $G$ is conformal and orientation preserving and therefore holomorphic. As a biholomorphic map between tori $G$ is of the form $G(z)=az+b$ with inverse $G^{-1}(z)= \frac{z}{a} - \frac{b}{a}$ and $a\Gamma_2=\Gamma_1$. Furthermore, $\tilde{I}=G^{-1}\circ I \circ G$ is an antiholomorphic fixpoint-free involution on $M_{r_2}$. Arguing as in Corollary~\ref{thmgcircI} using the natural involution $J: P \to P$ we deduce
\[ J\circ F_2 \circ \tilde{I}(z) = J \circ F_1 \circ I \circ G(z) = F_1\circ G(z) = F_2(z)= J \circ F_2 \circ I (z).\]
As $J$ is an involution and $F_2$ is (restricted to a fundamental domain) a homeomorphism onto its image we hence conclude $\tilde{I}=I$. 
By direct computation following \[G^{-1}\circ I \circ G(z) = \frac{\bar{a}}{a} \bar{z} + \frac{1}{2a} + \frac{\bar{b}-b}{a} = \bar{z} + \frac12\]
we deduce $a=1$ and $\Im(b)=0$. This implies $r_1=r_2$ and $F_1(z+b)=F_2(z)$ for some $b \in [0,1)$.
\end{proof}

\begin{corollary}\label{uniqueness up to rigid motions}
Let $f_i^{r_i}: K \to \R^4$, $i=1,2$, be a pair of Klein bottles with $\W(f_i^{r_i})=8\pi$ and $e(\nu_i^{r_i})=-4$ and $f_1^{r_1}(K)=\lambda R\circ f_2^{r_2}(K)$ for a rigid motion $R(x)=O(x+v)$, $x\in\R^4$, i.e.\ $O\in SO(4)$, $v\in\R^4$, and a scaling factor  $\lambda \in \R^+$. Then we have that $r_1=r_2$. Furthermore if $g_i: M_{r_i} \to \C P^1$ are the related projections of the holomorphic lifts, there is a M\"obius transform $m_O \in G$ of Proposition~\ref{prop: SO(4) action} and $b \in [0,1)$ such that $g_1(z)=m_O\circ g_2(z+b)$.
\end{corollary}
\begin{proof}
Firstly observe that scaling and rigid motions on $\R^4$ are conformal transformations. Corollary~\ref{cor:uniqueness of the conformal class} hence implies that $r_1=r_2$. We use the notation  $f_i\defi f^{r_i}_i$. \\
Secondly we can assume that $\lambda=1$ and $v=0$ i.e.\ $f_1(K)=O\cdot f_2(K)$ for some $O \in SO(4)$. Otherwise we may 
consider additionally the immersion $f'_2:= \lambda(f_2+v)$ and observe that 
\[ p\circ F_2(z) = p\circ F_2'(z) \text{ for all } z \in M_{r_2}, \] 
where $F_2, F_2': M_{r_2} \to H\oplus H$ are the related lifts and $p$ the projection $p: H \oplus H \to \C P^1$. This is easily seen because scaling and translation in $\R^4$ does not affect the tangent space and so the lift in the twistor space is unaffected.\\
Let $F_i: M_{r_i} \to H\oplus H$ be the lift of $f_i: K \to \R^4$ and $O \in SO(4)$ such that $f_1(K)=O\cdot f_2(K)$. The group $SO(4)$ acts naturally on $P$, compare Proposition~\ref{prop: SO(4) action}, hence $O\cdot F_2$ is the lift of $O\cdot f_2$. Using the proof of Corollary~\ref{cor:uniqueness of the conformal class} we conclude the existence of $b \in [0,1)$ such that 
\[ F_1(z) = O\cdot F_2(z+b) \text{ for all } z\in M_{r_1}=M_{r_2}.\]
Furthermore, Corollary~\ref{prop: SO(4) action} implies the existence of a M\"obius transform $m_O \in G$ such that $p\circ O\cdot F_2 = m_O \circ p \circ F_2$. With $g_i=p\circ F_i$ we get
\[ g_1(z) = m_O\circ g_2(z+b) \text{ for all } z \in M_{r_1}=M_{r_2}.\]
\end{proof}

\begin{remark}
Concerning the question how many surfaces we have found in Theorem~\ref{Unserthm} we can say the following: As shown in the proof of Lemma~\ref{lemmaexg} the parameter set (of $g_\lambda$ and therefore of $f^r_\lambda$) is at least of the size of $[0,1]^7$. In Corollary~\ref{uniqueness up to rigid motions} we studied how rigid motions and scaling in $\R^4$ and admissible reparametrizations of the tori affect our surfaces, in particular the ${g_\lambda}'s$. Counting dimensions we still have a parameter set $[0,1]^3$ for every torus $M_r$. 
 
 \end{remark}
 
 We finish this section with deducing the explicit formula for the double cover of the Veronese embedding. We repeat the statement from the introduction:

\begin{proposition}
 Define $f:  \Szw \to\C^2 = \R^4$ by $f(z)= \left(\bar z \frac{\abs{z}^4 - 1}{\abs{z}^6 + 1}, z^2 \frac{\abs{z}^2 +1}{\abs{z}^6 + 1}\right)$. Then $f\left(\Szw\right)$ is the Veronese surface (up to conformal transformation of $\R^4$).
\end{proposition}

\begin{proof}
 Consider the triple $(g,s_1,s_2)$ where $g=z^3$, $s_1=(\varphi^1,\varphi^2)= (z^2,\frac{1}{z})$, $s_2=(\psi^1,\psi^2)=(z, \frac{1}{z^2})$. Then $g$ satisfies $g\circ I=-\frac{1}{\bar g}$ for the antiholomorphic involution $z\mapsto -\frac{1}{\bar z}$ without fixpoints on $\Szw$. Furthermore, our choice of $s_1$ and $s_2$ yields \eqref{phicircI}. The immersion $f$ is defined such that $f(z)=\left(\frac{\varphi^1 \bar g - \bar\psi^1}{1 + |g|^2} , \frac{\bar\psi^1 g + \varphi^1}{1 + |g|^2}\right)= \left(\frac{\varphi^1-\varphi^1\circ I}{g - g\circ I},\frac{\bar\psi^1-\bar \psi^1\circ I}{\bar g - \bar g\circ I}\right)$, which is \eqref{weier} and \eqref{formf}. It follows that $f\circ I= f$. As $\abs{ds_1} + \abs{ds_2}>0$ on $\Szw$ we defined a twistor holomorphic immersion with $W(f)=12\pi$, see Theorem~\ref{weierstr} or \cite{Friedrich}. As $S^2$ carries the involution $I$ we consider $q: \Szw\to \QR{\Szw}{\langle I\rangle}$ and $\tilde f\defi f\circ q^{-1}: \QR{\Szw}{\langle I\rangle} \to \R^4$. We get an $\R P^2$ with $W(\tilde f)= 6\pi$. By the work of Li and Yau \cite{LiYau} $\tilde f$ must be a conformal transformation of the Veronese embedding.
\end{proof}

\section{A Klein bottle in $\R^4$ with Willmore energy less than $8\pi$} \label{Section3}

In this final section we would like to consider the case of immersions
$f:K\rightarrow \R^4$ with Euler normal number $e(\nu)=0$. Our goal is to show the existence of the minimizer of immersed Klein bottles in $\R^n$, $n\geq 4$. For this, we need the
following theorem:
\begin{theorem} \label{lesseightpi}
Let $K=\R P^2\#\R P^2$ be a Klein bottle.
Then there exists an embedding $f:K\rightarrow \R^4$ with $e(\nu)=0$ and $\W(f)<8\pi$.
\end{theorem}
For the proof of the preceding theorem we use a
construction by Bauer and Kuwert:
\begin{theorem} (M.~Bauer, E.~Kuwert \cite{BauerKuwert}). \label{connectedsum}
Let $f_i:\Sigma_i\rightarrow \R^n$, $i=1,2$, be two smoothly immersed, closed
surfaces. If neither $f_1$ nor $f_2$ is a round sphere (i.e.\ totally umbilical),
then there is an immersed surface $f:\Sigma\rightarrow \R^n$ with topological type of
the connected sum $\Sigma_1\#\Sigma_2$, such that \begin{align} \label{strictineq}
\W(f)<\W(f_1)+\W(f_2)-4\pi. \end{align}
\end{theorem}

\begin{remark}
 Notice that the strategy for the gluing construction implemented by Bauer and Kuwert was proposed by R.~Kusner \cite{Kusner3}.
\end{remark}

A rough sketch of the construction of the connected sum in Theorem \ref{connectedsum} is
as follows: The surface $f_1$ is inverted at an appropriate sphere in order to
obtain a surface $\hat{f_1}$ with a planar end and energy
$\W(\hat{f_1})=\W(f_1)-4\pi$. Then a small disk is deleted from $f_2$ and
a suitably scaled copy of $\hat{f_1}$ is implanted.
An interpolation yields the strict inequality (\ref{strictineq}). For the details
we refer to \cite{BauerKuwert}.
\\ \\
For the Veronese embedding $V:\R P^2\rightarrow \R^4$ we have $\W(V)=6\pi$. Hence
we can connect two Veronese surfaces and obtain a new surface $f$ with $\W(f)<8\pi$.
However, by the previous sections we know that there is no Klein bottle in $\R^4$
with Euler normal number $4$ or $-4$ and Willmore energy less than $8\pi$.
In order to obtain a better understanding of this situation we have to take a closer
look on the construction of Bauer and Kuwert: \\ \\
For that let $f_i:\Sigma_i\rightarrow \R^n=\R^{2+k}$, $i=1,2$, be two immersions
that are not totally umbilical (i.e.\ no round spheres). Let $A,B$ denote the
second fundamental forms of $f_1$, $f_2$ respectively. Moreover let $p_i\in \Sigma_i$
be two points, such that $A^\circ(p_1)$, $B^\circ(p_2)$ are both nonzero. After a translation
and a rotation we may assume \begin{align*}
f_i(p_i)=0, \hspace{1cm} \text{im}\;Df_i(p_i)=\R^2\times\{0\} \text{ for } i=1,2. \end{align*}
Then $A^\circ(p_1),B^\circ(p_2):\R^2 \times \R^2\rightarrow \R^k$ are symmetric, tracefree, nonzero
bilinear forms. \\ \\
In \cite{BauerKuwert}, p.\ 574, (4.34), it is shown that Theorem \ref{connectedsum} is true provided \begin{align}
\label{scalarproduct} \langle A^\circ(p_1),B^\circ(p_2)\rangle > 0. \end{align}
In order to achieve inequality (\ref{scalarproduct}) one exploits the freedom to rotate the surface $f_1$ by
an orthogonal transformation $R\simeq (S,T)\in \mathbb{O}(2)\times \mathbb{O}(k)\subset \mathbb{O}(n)$
before performing the connected sum construction. \\ \\
The second fundamental form $A_{S,T}$ of the rotated surface $Rf_1$ at the origin is given by
\begin{align*}
A_{S,T}(\zeta,\zeta)=TA(S^{-1}\zeta,S^{-1}\zeta) \hspace{1cm} \text{for all }\zeta\in \R^2.
\end{align*}
For the tracefree part we obtain
\begin{align*}
A_{S,T}^\circ(\zeta,\zeta)=TA^\circ(S^{-1}\zeta,S^{-1}\zeta) \hspace{1cm} \text{for all }\zeta\in \R^2.
\end{align*}
We need the following linear algebra fact that will be applied to $A^\circ$, $B^\circ$:
\begin{lemma} \label{lalemma}
Let $P,Q:\R^2\times\R^2\rightarrow \R^k$ be bilinear forms that are symmetric, tracefree and both nonzero.
\begin{itemize}
\item[a)] There exist orthogonal transformations $S\in \mathbb{SO}(2)$ and $T\in \mathbb{O}(k)$,
such that the form $P_{S,T}(\zeta,\zeta)=TP(S^{-1}\zeta,S^{-1}\zeta)$ satisfies
$\langle P_{S,T},Q\rangle > 0$.
\item[b)] We can choose $S\in\mathbb{SO}(2)$ and $T\in\mathbb{SO}(k)$ such that
$\langle P_{S,T},Q\rangle > 0$, except for the case that all of the following properties
are satisfied:
\begin{itemize}
\item[$\bullet$] $k=2$,
\item[$\bullet$] $|P(e_1,e_1)|^2=|P(e_1,e_2)|^2$\;\; and \;\;$|Q(e_1,e_1)|^2=|Q(e_1,e_2)|^2$,
\item[$\bullet$] $\langle P(e_1,e_1),P(e_1,e_2)\rangle=0$ \hspace{1.3mm} and \; $\langle Q(e_1,e_1),Q(e_1,e_2)\rangle=0$,
\item[$\bullet$] $\{P(e_1,e_1),P(e_1,e_2)\}$ and $\{Q(e_1,e_1),Q(e_1,e_2)\}$ determine opposite orientations of $\R^2$.
\end{itemize}
In this case, if $S\in \mathbb{SO}(2)$ and $T\in \mathbb{O}(2)$ with
$\langle P_{S,T},Q\rangle >0$, we have $T\in \mathbb{O}(2)\setminus \mathbb{SO}(2)$.
\end{itemize}
In b), the ordered set $\{e_1,e_2\}$ is any positively oriented orthonormal basis of $\R^2$.
If all four properties in b) are satisfied for one such basis, they are also satisfied for any other
positively oriented orthonormal basis of $\R^2$.
\end{lemma}
\begin{proof}
See \cite{BauerKuwert}, p.\ 574, Lemma 4.5 The exceptional case in b) is the case in \cite{BauerKuwert} in which $k=2$,
$|b|=a$, $|d|=c$ and $b$, $d$ have opposite signs. \end{proof}
\begin{proof}[Proof of Theorem \ref{lesseightpi}]
Let $V:\R P^2\rightarrow \R^4$ be the Veronese embedding. Consider two copies
$f_i:\Sigma_i\rightarrow \R^4$ of $V$, i.e.\ $\Sigma_i=\R P^2$ and $f_i=V$ for $i=1,2$.
Let $A,B$ denote the second fundamental forms of $f_1$, $f_2$ respectively. Of course,
$A=B$ and $A^\circ=B^\circ$. Let $p\in \R P^2$ such that
$A^\circ(p)$ is nonzero (in fact, this is satisfied for any $p\in \R P^2$)
and set $P:=A^\circ(p)$, $Q:=B^\circ(p)$. Then $P,Q$ are two
bilinear forms as in Lemma \ref{lalemma} with $P=Q$. Surely, the last condition in the exceptional
case of Lemma \ref{lalemma} b) fails to be true. Hence, by Lemma \ref{lalemma} we can rotate
$f_1$ by rotations $S,T\in \mathbb{SO}(2)$
such that (\ref{scalarproduct}) is satisfied. Now we are able to perform the connected sum construction:
Inverting $f_1$ and connecting $f_1$ and $f_2$ as described in \cite{BauerKuwert} yields
a surface $f:K\rightarrow \R^4$ with $e(\nu)=0$ and $\W(f)<8\pi$.
As any closed surface with Willmore energy less than $8\pi$ is injective, $f$ is an
embedding.
\end{proof}
Let us finally explain why it is not possible to construct an immersion $f:K\rightarrow\R^4$
with $|e(\nu)|=4$ with the method above:
A direct calculation shows that the Veronese embedding $V$ satisfies
$|A^\circ_{11}|^2=|A^\circ_{12}|^2=1$ and $\langle A^\circ_{11},A^\circ_{12}\rangle=0$ in any
point of $\R P^2$. Let $P,Q$ be defined as in the preceding paragraph.
Then $P$, $Q$ satisfy the second and the third condition of the exceptional case in Lemma 
\ref{lalemma}. In order to obtain
a surface with $|e(\nu)|=4$ we have to reflect one of the Veronese surfaces before
rotating $f_1$ and performing the gluing construction. But then, also the last condition
of the exceptional case in Lemma \ref{lalemma} b) is satisfied. Hence we cannot choose
$T\in \mathbb{SO}(2)$, i.e.\ $f_1$ has to be reflected another time. But then,
after inverting $f_1$ and connecting the surfaces,
 $e(\nu)=0$ for the new surface. Hence, in this very special case, the construction above fails.
\begin{remark} We can also argue the other way round: Theorem \ref{Wintgen} implies that we cannot
choose $S\in\mathbb{SO}(2)$ in Lemma \ref{lalemma}. This implies
$|A^\circ_{11}|^2=|A^\circ_{12}|^2$ and $\langle A^\circ_{11},A^\circ_{12}\rangle=0$ for the 
Veronese embedding. Moreover, the surface $f:K\rightarrow\R^4$ that we obtain from
Theorem \ref{connectedsum} must have Euler normal number $0$ as $\W(f)<8\pi$.
\end{remark}

\begin{remark}
 As we can add arbitrary dimensions to $\R^4$ we get by Theorem~\ref{lesseightpi} that every Klein bottle can be embedded into $\R^n$, $n\geq 4$, with $\W(f)<8\pi$.
\end{remark}

The existence of a smooth embedding $\tilde{f}_0: K \to \R^n$, $n \ge 4$, minimizing the Willmore energy in the class of all immersions $\tilde f: K \to \R^n$ can be deduced by a compactness theorem of E.~Kuwert and Y.~Li, \cite[Proposition~4.1, Theorem~4.1]{KuwertLi} and the regularity results of E.~Kuwert and R.~Sch\"atzle \cite{KuwertSchaetzle} or T.~Rivi{\`e}re \cite{Riviere2, Riviere} if one can rule out diverging in moduli space. We note that T.~Rivi{\`e}re showed independently a compactness theorem similar to the one of Kuwert and Li, see \cite{Riviere}.
The non-de\-gene\-rating property is shown combining the subsequent Theorem~\ref{theo:divergimoduli} and Theorem~\ref{lesseightpi}. We get the following theorem which is Theorem~\ref{thm0}:

 \begin{theorem} 
  Let $S$ be the class of all immersions $f:\Sigma \to\R^n$ where $\Sigma$ is a Klein bottle. Consider
 \begin{align*}
  \beta_2^n\defi \inf \left\{\W(f): f\in S\right\}.
 \end{align*}
Then we have that $\beta_2^n< 8\pi $ for $n\geq 4$. Furthermore, $\beta_2^n$ is attained by a smooth embedded Klein bottle for $n\geq 4$.
\end{theorem} 

Before we prove that a sequence of degenerating Klein bottles always has $8\pi$ Willmore energy we explain how we apply certain techniques from \cite{KuwertLi} to non-orientable closed surfaces. \\
We repeat our general set-up from the beginning of Section~\ref{Section1}: Let $N$ be a non-orientable closed manifold of dimension two and $\tilde f: N\to\Rn$ ($n \geq 3$) an immersion. 
Consider $q:M\to N$, the conformal oriented 
two-sheeted cover of $N$, and define $f\defi \tilde f\circ q$. 
As every $2-$dimensional 
oriented manifold can be locally conformally reparametrized $M$ is a Riemann surface that is conformal to $(M, f^\ast \delta_{\text{eucl}})$.
Let $I:M\to M$ 
be the antiholomorphic order two deck transformation for $q$. The map $I$ is an 
antiholomorphic involution without fixpoints such that $ f\circ I = f$.
From now on we will work with the immersion $f$ on the Riemann surface $M$ equipped with an antiholomorphic involution $I$. We are not arguing on the quotient space $N= \QR{M}{\langle I \rangle}$.\\
For the Willmore energy of the immersion $f$ we have: 
\[ \W(f) = 2 \W(\tilde{f}).\]
If $p \in f^{-1}(y)$ then $I(p) \in f^{-1}(y)$ i.e.\ the number of pre-images of $f$ is always even. We describe this in other words: Consider $M$ as a varifold and consider the push-forward of $M$ via $f$ i.e.\ $f_\sharp M$. Then $f_\sharp M$ is a compactly supported  rectifiable varifold with at least multiplicity $2$ at every point.\\
We now consider the case that $f$ is a proper branched conformal immersion, \label{pres} compare \cite[page 323]{KuwertLi} i.e.\ there exists $\Sigma \subset M$ discrete such that $f \in W^{2,2}_{conf, loc}(M\setminus \Sigma, \R^n)$ and 
\[ \int_{U} \abs{A}^2 d\mu_{f^*\delta_{eucl}} < \infty \text{ and } \mu_{f^*\delta_{eucl}}(U)< \infty \text{ for all } U \subset \subset M. \]
We note once again that  we have $I(\Sigma) = \Sigma$ since $f\circ I = f$. If $\varphi: B_\sigma \to M$ is a local conformal parametrization around $\varphi(0) \in \Sigma$ such that $\varphi(B_\sigma) \cap \Sigma = \varphi(0)$ we may apply the classification of isolated singularities result of Kuwert and Li, \cite[Theorem 3.1]{KuwertLi}, to $f\circ \varphi$ and conclude that 
\[ \theta^2( f\circ \varphi_\sharp \a{B_\sigma}, f\circ \varphi(0)) = m +1 \text{ for some } m \ge 0. \]
In here, we considered $\a{B_\sigma}$ as an varifold itself. 
Furthermore $I\circ\varphi : B_\sigma \to M$ is an antiholomorphic parametrization around the point $I\circ\varphi(0)$. Applying once more \cite[Theorem~3.1]{KuwertLi} ($I\circ \varphi(B_\sigma)\cap \varphi(B_\sigma)=\emptyset$ by the choice of $\sigma$) 
\[ \theta^2( f\circ I\circ \varphi_\sharp \a{B_\sigma}, f\circ I\circ \varphi(0)) = m' +1 \text{ for some } m' \ge 0.\]
We have $m=m'$ since $f\circ I = f$. Combining both local estimates with the monotonicity formula of Simon (that extends to branched conformal immersions) we obtain for $q=f\circ\varphi(0) = f\circ I\circ \varphi(0)$. 
\begin{align}\begin{split}
 \W(f) &\ge \theta^2(f_\sharp M, q) \\
 &\ge \theta^2( f\circ \varphi_\sharp \a{B_\sigma}, q)+ \theta^2( f\circ I \circ \varphi_\sharp \a{B_\sigma}, q)\\
 & \ge 2(m+1) \,4\pi.\end{split} \label{unglLiYau}
\end{align}

We remark that in general we could have started working on $N= \QR{M}{\langle I \rangle}$ with the associated varifold $\tilde{f}_\sharp N$ which has density $1$ at most points. But we decided to stick to the oriented double cover $M$ since all theorems in the literature are proven on orientable Riemann surfaces. \\

The following theorem can be considered as the analog of \cite[Theorem~5.2]{KuwertLi} for the non-orientable situation. Our argumentation is inspired by the arguments of Kuwert and Li. 

\begin{theorem}\label{theo:divergimoduli}
Let $K_m$ be a sequence of Klein bottles diverging in moduli space. Then for any sequence of conformal immersions $\tilde{f}_m\in W^{2,2}_{\conf}(K_m, \R^n)$ we have 
\[ \liminf_{m \to \infty} \W(\tilde{f}_m) \ge 8\pi.\]
\end{theorem}

\begin{proof}
Let $q_m: T^2_m \to K_m$ be the two sheeted oriented double cover and $I_m: T^2_m \to T^2_m$ the associated antiholomorphic order two deck transformation. By Theorem~\ref{thminvolution} we may assume that $T^2_m=\QR{\C}{\Gamma_m}$, where $\Gamma_m$ is a lattice generated by $(1, ib_m)$ with $b_m \ge 1$ and $I_m$ is given by
\begin{align}\label{eq:invol}
 \text{either } I_m(z)= \bar{z} + \frac12 \text{ or } I_m(z)= - \bar{z} + i \frac{b_m}{2}.
\end{align}
Diverging in moduli space implies $\lim_{m \to \infty} b_m = \infty$.
We lift the maps $\tilde{f}_m$ to the double cover $T^2_m$ and then to $\Gamma_m$-periodic maps from $\C$ into $\R^n$ and denote the lifted maps by $f_m$, i.e.\ $f_m\circ I_m = f_m$. By Gauss-Bonnet we may also assume   that the maps $f_m: \C \to \R^n$ satisfy 
\[ \limsup_{m \to \infty} \frac14 \int_{T^2_m} \abs{A_{f_m}}^2 \, d\mu_{g_m} = \limsup_{m\to \infty} \W(f_m) \le W_0 < \infty. \]
The theorem is proven if we show that 
\begin{equation}\label{eq:16pibound}
\liminf_{m \to \infty} \W(f_m) \ge 16 \pi.
\end{equation}
We have to distinguish two cases. They are determined by the form of the involution. After passing to a subsequence the involution is either of the second kind in \eqref{eq:invol} for all $m$ (Case 1) or it is the involution $I(z)=\bar z + \frac12$ for all $m$ (Case 2).\\
Following the notation of \cite{KuwertLi} we will denote by $I_{p}$ the inversion at $\partial B_1(p)$ in $\R^n$, i.e.\ $I_p(x)=p + \frac{ x-p}{\abs{x-p}^2}$ for $x\in \R^n$. Furthermore, for $\delta \in \R$ we define the translations $\eta_\delta(z):=z + i \delta$ for $z\in\C$.\\
\emph{Case 1:} $I_m(z)= -\bar{z} + i \frac{b_m}{2}$ for all $m$. \\
\emph{Proof of \eqref{eq:16pibound} in Case 1:} The local $L^\infty$-bound of the conformal factor \cite[Corollary~2.2]{KuwertLi} implies that $f_m$ is not constant on any circle $C_v=[0,1]\times\{v\}$. As $f_m\circ I_m = f_m$ we have that $f_m( [0,1] \times [0, \frac{b_m}{2}]) = f_m( [0,1] \times [\frac{b_m}{2}, b_m])$. Thus, there exists $v_m \in [0, \frac{b_m}{2})$ s.t.\ \[ \lambda_m:=\diam ( f_m(C_{v_m})) \le \diam ( f_m(C_{v})) \text{ for all } v \in \R. \]
As already mentioned in Lemma~\ref{lemma3} the involution is not affected by these translations because $\eta_\delta^{-1}\circ I_m \circ \eta_\delta = I_m - 2 \Re( i \delta) = I_m$. Consider the two sequences 
\begin{align*} h_m(z)&= \lambda_m^{-1} \left( f_m\circ\eta_{v_m}(z) - f_m\circ\eta_{v_m}(0) \right) \text{ and } \\
k_m(z)&= \lambda_m^{-1} \left( f_m\circ\eta_{\frac{b_m}{2} + v_m}(z) - f_m\circ\eta_{\frac{b_m}{2} + v_m}(0) \right).
\end{align*}
We have that $1=\diam ( h_m(C_0))= \diam (k_m(C_0))$, $0 =h_m(0)=k_m(0)$ for all $m$ and $h_m(z)=k_m(-\bar{z})$. The immersions $h_m$ and $k_m$ are immersed tori diverging in moduli space. We can therefore repeat the proof of E.~Kuwert and Y.~Li from \cite[Theorem~5.2]{KuwertLi}. We find a suitable inversion $I_{x_0}$ at a  sphere $\partial B_1(x_0)$ and deduce that $\hat{h}_m\defi I_{x_0} \circ h_m$, $\hat{k}_m \defi I_{x_0}\circ k_m$   converge locally uniformly to branched conformal immersions $\hat{h}$ and $\hat{k}$ satisfying $\W(\hat{h}) \ge 8\pi$ and $\W(\hat{k})\ge 8\pi$. Observe that 
\begin{align*}
&\W(f_m)= \W(I_{x_0}\circ f_m)\\
&=\frac14 \int_{[0,1] \times [ v_m - \frac{b_m}{4}, v_m + \frac{b_m}{4}]}  \abs{H_{I_{x_0}\circ f_m}}^2 d\mu_{\hat{g}_m} + \frac14 \int_{[0,1] \times [ v_m + \frac{b_m}{4}, v_m + \frac{3b_m}{4}] } \abs{H_{I_{x_0}\circ f_m}}^2 d\mu_{\hat{g}_m}\\
&= \W\left(\hat{h}_m|_{[0,1]\times[-\frac{b_m}{4}, \frac{b_m}{4}]}\right)+\W\left(\hat{k}_m|_{[0,1]\times[-\frac{b_m}{4}, \frac{b_m}{4}]}\right).
\end{align*}
We pass to the limit and get
\[ \liminf_{m \to \infty} \W(f_m) \ge \liminf_{m \to \infty} \W(\hat{h}_m) + \liminf_{m\to \infty} \W(\hat{k}_m) \ge \W(\hat{h})+\W(\hat{k}) \ge 16 \pi.\]
Note that $\hat h$ and $\hat k$ parametrize the same sphere because of $\hat{h}(z)=\hat{k}(-\bar{z})$. This sphere has a double point as shown in the proof of \cite[Theorem~5.2]{KuwertLi}.\\
\emph{Case 2:} $I_m(z)=\bar{z} + \frac{1}{2}$ for all $m$. \\
\emph{Proof of \eqref{eq:16pibound} in Case 2:} Observe that we cannot translate into ``imaginary direction'' without changing the involution because $\eta_{\delta}^{-1} \circ I \circ \eta_\delta(z) = I(z) - 2 i \delta$. Another delicate point is that the form of the involution does not help to find a ``second'' torus.\\
We fix a large integer $M\in \N$ such that $4\pi M \ge W_0$.\\
For each $m\in \N$ pick $u_m \in \{-M, ..., M \}$ such that
\begin{equation}\label{eq:diamfirstrp2} \lambda_m=\diam(f_m(C_{u_m}))= \min_{u \in\{ -M,...,M\}}\diam(f_m(C_u)).\end{equation}
By passing to a subsequence we may assume that $u_m=u_0$ for all $m$. Furthermore, arguing as in \cite[Propositon~4.1]{KuwertLi} we obtain $B_1(x_1)\subset \R^n$ such that $f_m(T^2_m) \cap B_1(x_1)= \emptyset$ for all $m$. We consider the sequence
\begin{equation}\label{eq:firstrp2} h_m(z) \defi I_{x_1}\left( \lambda_m^{-1}( f_m(z) - f_m\circ\eta_{u_0}(0))\right). \end{equation}
Repeat the procedure and fix $v_m \in \{-M, ..., M \}$ such that
\begin{equation}\label{eq:diamsecondrp2} \mu_m=\diam(h_m(C_{\frac{b_m}{2}+v_m}))= \min_{v \in\{ -M,...,M\}}\diam(h_m(C_{\frac{b_m}{2}+v})).\end{equation}
By passing to a subsequence we may assume $v_m=v_0$ for all $m$ and define
\[ k_m(z) \defi \mu_m^{-1} \left( h_m \circ \eta_{ \frac{b_m}{2}}(z) - h_m \circ \eta_{ \frac{b_m}{2} + v_0}(0)\right).\]
The translations were chosen such that we still have $h_m\circ I=h_m$ and $k_m\circ I = k_m$ for all $m$.
As before we find $x_2 \in \R^n$ with $k_m(T^2_m)\cap B_1(x_2) = \emptyset$ and consider $\hat{k}_m = I_{x_2}(k_m)$.
We have achieved that $h_m(T^2_m) \subset \overline{B_1(x_1)}$ and $\hat{k}_m(T^2_m) \subset \overline{B_1(x_2)}$. Lemma~1.1 from \cite{Simon} implies  area bounds $\mu_{g_m}(T^2_m)\leq C$ for both sequences. Up to a subsequence, we have $\abs{A_{h_m}}^2\, d\mu_{g_m} \to \alpha_1$ and $\abs{A_{\hat{k}_m}}^2\, d\mu_{g_m} \to \alpha_2$ as Radon measures on the cylinder $C=[0,1]\times \R$. The sets $\Sigma_i:=\{z \in \C \colon \alpha_i(\{z\}) \ge 4\pi \}$, $i=1,2$ are discrete. Theorem~5.1 in \cite{KuwertLi} yields that $h_m$ and $\hat{k}_m$ converge locally uniformly on $C \setminus \Sigma_1$ and $C\setminus \Sigma_2$, respectively. The limits either are conformal immersions $h: C \setminus \Sigma_1 \to \R^n$, $\hat{k}: C \setminus \Sigma_2\to\R^n$ or points $p_1, p_2$. Note that by construction
\begin{align}\label{eq:distanceorigine}
h_m(C_{u_0})\subset I_{x_1}(\overline{B_1(0)}) \subset \R^n \setminus B_{\theta_1}(x_1) \text{ with } \theta_1=\frac{1}{1+\abs{x_1}} \\ \nonumber
\hat{k}_m(C_{v_0})\subset I_{x_2}(\overline{B_1(0)}) \subset \R^n \setminus B_{\theta_2}(x_2) \text{ with } \theta_2=\frac{1}{1+\abs{x_2}}
\end{align}
Assume the second alternative holds for $h_m$ i.e.\ $h_m \to p_1$ locally uniformly. Observe that $C_u \cap \Sigma_1 = \emptyset$ for at least one $u_{*}\in\{-M,...,M\}$. Otherwise there would be points $z_u \in C_u$ with $\alpha_1(B_{\frac14}(z_u))> 4\pi$ for each $u \in \{-M, ..., M \}$ contradicting $\sum_{u=-M}^M \alpha_1(B_{\frac{1}{4}}(z_u)) \le \alpha_1(C) \le W_0$. \\
Due to \eqref{eq:distanceorigine} we have $\abs{p_1-x_1} \ge \theta_1> 0$ and hence $I_{x_1}(h_m(C_{u_*}))\to I_{x_1}(p_1)$ uniformly. But $\diam(I_{x_1}(h_m(C_{u_*})))\geq 1$ by \eqref{eq:diamfirstrp2} and \eqref{eq:firstrp2}, a contradiction. In the same way we exclude $\hat{k}_m \to p_2$.\\ 
By uniform local convergence we get 
\begin{align*}
 h & =\lim_{m\to\infty}h_m=\lim_{m\to\infty} h_m\circ I = h\circ I,\\
 \hat k&=\lim_{m\to\infty}\hat k_m=\lim_{m\to\infty} \hat k_m\circ I = \hat k\circ I,
\end{align*}
and we are in the situation of branched $W^{2,2}$-conformal immersions that are invariant under $I$. 
We now investigate the behavior of $h, \hat{k}$ at the ends $\{\pm \infty\}$ of the cylinder $C$. 
We present the argument for $h$, the argument for $\hat{k}$ works analogously. 
We note that $\varphi_+(z):=\frac{-i}{2\pi} \ln(z)$ is a holomorphic chart around $+\infty$ and $I\circ\varphi_+(z)= \frac{i}{2\pi} \ln(\bar{z})+ \frac12$ is an antiholomorphic chart around $-\infty$.  Since $\int_{C} \abs{A_h}^2 \, d\mu_g \le \alpha_1(C)< \infty$ the map $h_{+}(z):=h\circ \varphi_+(z)$ is a $W^{2,2}_{loc}(B_\sigma\setminus\{0\}, \R^n)$-conformal immersion with $\Sigma_1 \cap \varphi_+(B_\sigma\setminus\{0\}) = \emptyset$ for $\sigma>0$ sufficiently small. 
We follow the explanations presented in front of Theorem~\ref{theo:divergimoduli} (p.~\pageref{pres}) to conclude that the varifold ${h_+}_\sharp \a{B_\sigma}$ extends continuously to $0$. This implies that $h(C_v) \to q_1$ for $\nu \to \pm \infty$ using the fact that $I(C_v)=C_{-v}$. 
Furthermore, applying the Li-Yau inequality (a version for branched immersion can be found in \cite[formula~(3.1)]{KuwertLi}) yields the following lower bound of the Willmore energy of $h$. A detailed explanation how we apply this inequality to the oriented double covers was done on page~\pageref{pres}, see (\ref{unglLiYau}). We have that
\begin{align*}
 \W(h) &\ge \theta^2(h_\sharp C, q_1) \\
 &\ge \theta^2( {h_+}_\sharp \a{B_\sigma}, q)+ \theta^2( {h\circ I \circ \varphi_+}_\sharp \a{B_\sigma}, q)\\
 & \ge 2(m(+\infty)+1) \, 4\pi. 
\end{align*}
Thus, if $m(+\infty)\ge 1$ we have that $\W(h) \ge 4 \cdot 4\pi$.
The very same argument applies to $\hat{k}$.\\
Similarly we exclude branch points for the maps $h, \hat{k}$ in the interior  of the cylinder $C$ as follows. Suppose that the application of the classification theorem of isolated singularities \cite[Theorem~3.1]{KuwertLi} to a point $z \in \Sigma_1$ reveals a point with branching order $m(z) \ge 1$ then by \eqref{unglLiYau} we conclude $\W(h)\ge 4 \cdot 4\pi = 16 \pi$. In the same way we can assume that all points $z \in \Sigma_2$ are removable singularities, i.e.\ $m(z)=0$ and $\hat{k}$ has removable singularities in $\pm \infty$. \\
It remains the situation where $h$ and $\hat{k}$ are unbranched. Since $h\circ I =h$, $\hat{k}\circ I =\hat k$ and $h, \hat{k}$ extend smoothly to $\pm \infty$ they are double covers of immersions of $\R P^2$'s into $\R^n$. By the work of Li and Yau \cite{LiYau} we get $\min\{ \W(h), \W(\hat{k})\} \ge 2\cdot 6 \pi= 12 \pi$ as $h$ and $\hat k$ are the oriented double covers of the unoriented surfaces. Recall once more that this also implies that $\# h^{-1}(\{x\})$ and $\# \hat{k}^{-1}(\{x\})$ are even for all $x \in \R^n$. \\
If $\# \hat{k}^{-1}(\{x_1\}) >2 $ the Li-Yau inequality implies $\W(\hat{k}) \ge  4\pi \cdot \# \hat{k}^{-1}(\{x_1\})  \ge 16 \pi$. Otherwise let $C_m:=[0,1]\times [-\frac{b_m}{4}, \frac{b_m}{4}]$. We observe that $k_m= I_{x_2}\circ \hat{k}_m \to I_{x_2}\circ \hat{k}$ and 
\[ \W(h_m) = \W(h_m|_{C_m}) + \W(h_m|_{\eta_{ \frac{b_m}{2}}(C_m)}) = \W(h_m|_{C_m}) + \W(k_m|_{C_m}).\]
With $\#\hat{k}(x_1)\le 2$ we conclude by \cite[Formula~(3.1)]{KuwertLi} that 
\begin{align*} \liminf_{m \to \infty} W(h_m) &\ge \liminf_{m \to \infty} \W(h_m|_{C_m}) + \liminf_{m \to \infty}\W(k_m|_{C_m})\\
&\ge \W(h) + \W(I_{x_1}(\hat{k})) \ge 12 \pi + (12 \pi - 8 \pi) = 16 \pi.
\end{align*}

\end{proof}

\bibliographystyle{plain}
\bibliography{Lit}

\end{document}